\documentclass{amsart}
\usepackage{hyperref,url}
\newtheorem{theorem}{Theorem}[section]
\newtheorem{lemma}[theorem]{Lemma}

\newcommand{\C}{\mathbb{C}}
\newcommand{\R}{\mathbb{R}}
\newcommand{\Z}{\mathbb{Z}}
\newcommand{\cL}{\mathcal{L}}
\newcommand{\cR}{\mathcal{R}}
\newcommand{\eps}{\varepsilon}
\begin{document}



\title[Singular integral operators]{%
Singular integral operators\\
on Nakano spaces with weights\\
having finite sets of discontinuities}


\author{Alexei Yu. Karlovich}
\address{
Departamento de Matem\'atica,
Faculdade de Ci\^encias e Tecnologia\\
Universidade Nova de Lisboa,
Quinta da Torre,
2829--516 Caparica,
Portugal\\
E-mail: oyk@fct.unl.pt}

\dedicatory{To the memory of Professor Israel Gohberg (23.08.1928--12.10.2009)}

\begin{abstract}
In 1968, Gohberg and Krupnik found a Fredholm criterion for
singular integral operators of the form $aP+bQ$,
where $a,b$ are piecewise continuous functions and $P,Q$ are
complementary projections associated to the Cauchy singular
integral operator, acting on Lebesgue spaces over Lyapunov curves.
We extend this result to the case of Nakano spaces (also known as
variable Lebesgue spaces) with certain weights having finite sets
of discontinuities on arbitrary Carleson curves.
\end{abstract}
\maketitle
\section{Introduction}
We say that a rectifiable curve $\Gamma$ in the complex plane is
simple if it is homeomorphic to a segment or to a circle. We equip
$\Gamma$ with Lebesgue length measure $|d\tau|$. The
\textit{Cauchy singular integral} of $f\in L^1(\Gamma)$ is defined
by
\[
(Sf)(t):=\frac{1}{\pi i}\int_\Gamma\frac{f(\tau)}{\tau-t}d\tau
\quad(t\in\Gamma).
\]
This integral is understood in the principal value sense, that is,
\[
\int_\Gamma\frac{f(\tau)}{\tau-t}d\tau:=
\lim_{R\to 0}
\int_{\Gamma\setminus\Gamma(t,R)} \frac{f(\tau)}{\tau-t}d\tau,
\]
where $\Gamma(t,R):=\{\tau\in\Gamma:|\tau-t|<R\}$ for $R>0$. David
\cite{David84} (see also \cite[Theorem~4.17]{BK97}) proved that
the Cauchy singular integral generates the bounded operator $S$ on
the Lebesgue space $L^p(\Gamma)$, $1<p<\infty$, if and only if
$\Gamma$ is a \textit{Carleson} (\textit{Ahlfors-David regular})
\textit{curve}, that is,
\[
\sup_{t\in\Gamma}\sup_{R>0}\frac{|\Gamma(t,R)|}{R}<\infty,
\]
where for any measurable set $\Omega\subset\Gamma$ the symbol
$|\Omega|$ denotes its measure.

A measurable function $w:\Gamma\to[0,\infty]$ is referred to as a
\textit{weight function} or simply a \textit{weight} if
$0<w(\tau)<\infty$ for almost all $\tau\in\Gamma$. Suppose
$p:\Gamma\to[1,\infty]$ is a measurable a.e. finite function.
Denote by $L^{p(\cdot)}(\Gamma,w)$ the set of all measurable
complex-valued functions $f$ on $\Gamma$ such that
\[
\int_\Gamma |f(\tau)w(\tau)/\lambda|^{p(\tau)}|d\tau|<\infty
\]
for some $\lambda=\lambda(f)>0$. This set becomes a Banach space
when equipped with the Luxemburg-Nakano norm
\[
\|f\|_{p(\cdot),w}:=\inf\left\{\lambda>0: \int_\Gamma
|f(\tau)w(\tau)/\lambda|^{p(\tau)}|d\tau|\le 1\right\}.
\]
If $p$ is constant, then $L^{p(\cdot)}(\Gamma,w)$ is nothing else
but the weighted Lebesgue space. Therefore, it is natural to refer
to $L^{p(\cdot)}(\Gamma,w)$ as a \textit{weighted generalized
Lebesgue space with variable exponent} or simply as a
\textit{weighted variable Lebesgue space}. This is a special case
of Musielak-Orlicz spaces \cite{Musielak83} (see also
\cite{KR91}). Nakano \cite{Nakano50} considered these spaces
(without weights) as examples of so-called modular spaces, and
sometimes the spaces $L^{p(\cdot)}(\Gamma,w)$ are referred to as
weighted Nakano spaces.
\begin{theorem}[Kokilashvili, Paatashvili, S.~Samko]
\label{th:KPS} Suppose $\Gamma$ is a simple rectifiable curve and
$p:\Gamma\to(1,\infty)$ is a continuous function satisfying the
Dini-Lipschitz condition
\begin{equation}\label{eq:Dini-Lipschitz}
|p(\tau)-p(t)|\le -C_\Gamma/\log|\tau-t| \quad\mbox{whenever}\quad
|\tau-t|\le1/2,
\end{equation}
where $C_\Gamma$ is a positive constant depending only on
$\Gamma$. Let $t_1,\dots,t_n\in\Gamma$ be pairwise distinct points
and $\lambda_1,\dots,\lambda_n\in\R$. The Cauchy singular integral operator $S$
is bounded on the Nakano space $L^{p(\cdot)}(\Gamma,w)$ with
weight given by
\begin{equation}\label{eq:Khvedelidze-weight}
w(\tau)=\prod_{j=1}^n |\tau-t_j|^{\lambda_j}\quad(\tau\in\Gamma)
\end{equation}
if and only if $\Gamma$ is a Carleson curve and
$0<1/p(t_j)+\lambda_j<1$ for all $j\in\{1,\dots,n\}$.
\end{theorem}
For the case of constant $p$ and sufficiently smooth curves, the
sufficiency portion of the above result was obtained more than
fifty years ago by Khvedelidze \cite{Khvedelidze56} (see also
\cite[Chap.~1, Theorem~4.1]{GK92}). The
necessity portion for constant $p$ goes back to Gohberg and
Krupnik \cite{GK70}. For the complete solution of the
boundedness problem for the operator $S$ on weighted standard
Lebesgue spaces $L^p(\Gamma,w)$ we refer to the survey paper by
Dynkin \cite{Dynkin87}, to the monographs by B\"ottcher and Yu.
Karlovich \cite{BK97}, by Khuskivadze, Kokilashvili, and
Paatashvili \cite{KKP98}, and by Genebashvili, Gogatishvili,
Kokilashvili, and Krbec \cite{GGKK98}.

Theorem~\ref{th:KPS} was proved in \cite[Theorem~A]{KPS06}. Later on, Kokilashvili,
N.~Samko, and S.~Samko \cite[Theorem~4.3]{KSS07-MN}
generalized the sufficiency portion of Theorem~\ref{th:KPS} to the case
of radial oscillating weights
\begin{equation}\label{eq:radial-weight}
w(\tau)=\prod_{j=1}^n \omega_j(|\tau-t_j|) \quad(\tau\in\Gamma),
\end{equation}
where $\omega_j:(0,|\Gamma|]\to(0,\infty)$ are some continuous
functions oscillating at zero. Those sufficient boundedness
conditions are expressed in terms of the Matuszewska-Orlicz
indices \cite{MO60,MO65} (see also \cite{M85,M89}) of the
functions $\omega_j$. The author observed that those conditions
are also necessary for the boundedness of the operator $S$ on the
weighted Nakano space $L^{p(\cdot)}(\Gamma,w)$ in the case of
Jordan curves $\Gamma$ (see \cite[Corollary~4.3]{Karlovich09-IWOTA07}
and also \cite{Karlovich09-JFSA}). Recall that a rectifiable curve
in the complex plane is said to be Jordan if it is homeomorphic to
a circle.

Now fix $t\in\Gamma$ and assume that $w$ is a weight such that the
operator $S$ is bounded on $L^{p(\cdot)}(\Gamma,w)$. In the
spectral theory of one-dimensional singular integral operators it
is important to know whether the operator $S$ is also bounded on
the space $L^{p(\cdot)}(\Gamma,\varphi_{t,\gamma}w)$, where
\[
\varphi_{t,\gamma}(\tau):=|(\tau-t)^\gamma|
\]
and $\gamma$ is an arbitrary complex number. For standard Lebesgue
spaces and arbitrary Muckenhoupt wights such $\gamma$ are
completely characterized by B\"ottcher and Yu. Karlovich
\cite[Chap.~3]{BK97}. Notice that if $\gamma$ is the imaginary
unit, then $\varphi_{t,i}$ coincides with
\[
\eta_t(\tau):=e^{-\arg(\tau-t)}
\]
(here and in what follows we choose a continuous brunch of the
argument on $\Gamma\setminus\{t\}$), and this function lies beyond
the class of radial oscillating weights considered by
Kokilashvili, N.~Samko, and S.~Samko \cite{KSS07-JFSA,KSS07-MN}.
The author \cite[Theorem~2.1]{Karlovich10-IWOTA08} found necessary and
sufficient conditions for the boundedness of the operator $S$ on
the space $L^{p(\cdot)}(\Gamma,\varphi_{t,\gamma})$.

Our first aim in this paper is to generalize known boundedness
results for the operator $S$ on the space $L^{p(\cdot)}(\Gamma,w)$
to the case of weights of the form $w(\tau)=\prod_{j=1}^n
\psi_j(\tau)$ where each $\psi_j$ is a continuous positive
function on $\Gamma\setminus\{t_j\}$ and $t_1,\dots,t_n\in\Gamma$
are pairwise distinct points. In particular, we allow functions
$\psi_j$ of the form
$\psi_j(\tau)=(\eta_{t_j}(\tau))^x\omega_j(|\tau-t_j|)$ where
$x\in\R$ and $\omega_j$ is an oscillating function as in
\cite{Karlovich09-IWOTA07,Karlovich09-JFSA,KSS07-JFSA,KSS07-MN}.

To formulate our first main result explicitly, we need some
definitions. Following \cite[Section~1.4]{BK97}, a function
$\varrho:(0,\infty)\to(0,\infty]$ is said to be \textit{regular}
if it is bounded from above in some open neighborhood of the point
$1$. A function $\varrho:(0,\infty)\to(0,\infty]$ is said to be
\textit{submultiplicative} if $\varrho(xy)\le\varrho(x)\varrho(y)$
for all $x,y\in(0,\infty)$. Clearly, if
$\varrho:(0,\infty)\to(0,\infty]$ is regular and
submultiplicative, then $\varrho(x)$ is finite for all
$x\in(0,\infty)$. Given a regular submultiplicative function
$\varrho:(0,\infty)\to(0,\infty)$, one defines
\begin{equation}\label{eq:indices-definition}
\alpha(\varrho):=\sup_{x\in(0,1)}\frac{\log\varrho(x)}{\log x},
\quad
\beta(\varrho):=\inf_{x\in(1,\infty)}\frac{\log\varrho(x)}{\log x}.
\end{equation}
One can show (see Theorem~\ref{th:indices}) that
$-\infty<\alpha(\varrho)\le\beta(\varrho)<+\infty$. Thus it is natural to
call $\alpha(\varrho)$ and $\beta(\varrho)$ the lower and upper indices of
$\varrho$, respectively.

Fix $t\in\Gamma$ and $d_t:=\max\limits_{\tau\in\Gamma}|\tau-t|$.
Following \cite[Section~1.5]{BK97}, for every continuous function
$\psi:\Gamma\setminus\{t\}\to(0,\infty)$, we define
\[
(W_t\psi)(x):=\left\{\begin{array}{lll}
\displaystyle
\sup_{0<R\le d_t}\left(\max_{\tau\in\Gamma:|\tau-t|=xR}\psi(\tau)/\min_{\tau\in\Gamma:|\tau-t|=R}\psi(\tau)\right)
&\mbox{for}& x\in(0,1],
\\
\displaystyle
\sup_{0<R\le d_t}\left(\max_{\tau\in\Gamma:|\tau-t|=R}\psi(\tau)/\min_{\tau\in\Gamma:|\tau-t|=x^{-1}R}\psi(\tau)\right)
&\mbox{for}& x\in[1,\infty)
\end{array}\right.
\]
and
\begin{eqnarray*}
(W_t^0\psi)(x) &=&\limsup_{R\to 0}
\left(\max_{\tau\in\Gamma:|\tau-t|=xR}\psi(\tau)/\min_{\tau\in\Gamma:|\tau-t|=R}\psi(\tau)\right)
\\
&=&\limsup_{R\to 0}
\left(\max_{\tau\in\Gamma:|\tau-t|=R}\psi(\tau)/\min_{\tau\in\Gamma:|\tau-t|=x^{-1}R}\psi(\tau)\right)
\end{eqnarray*}
for $x\in\R$.
The function $W_t\psi$ is always submultiplicative. Moreover, if $W_t\psi$ is
regular, then $W_t^0\psi$ is also regular and submultiplicative and
\[
\alpha(W_t\psi)=\alpha(W_t^0\psi),\quad\beta(W_t\psi)=\beta(W_t^0\psi)
\]
(see \cite[Lemmas~1.15 and 1.16]{BK97}). Our first main result is
the following.
\begin{theorem}\label{th:boundedness}
Suppose $\Gamma$ is a simple rectifiable curve and
$p:\Gamma\to(1,\infty)$ is a continuous function satisfying the
Dini-Lipschitz condition {\rm(\ref{eq:Dini-Lipschitz})}. Let
$t_1,\dots,t_n\in\Gamma$ be pairwise distinct points and
$\psi_j:\Gamma\setminus\{t_j\}\to(0,\infty)$ be continuous
functions such that the functions $W_{t_j}\psi_j$ are regular for
all $j\in\{1,\dots,n\}$.
\begin{enumerate}
\item[{\rm(a)}]
If $\Gamma$ is a simple Carleson curve and
\begin{equation}\label{eq:boundedness-condition}
0<1/p(t_j)+\alpha(W_{t_j}^0\psi_j), \quad
1/p(t_j)+\beta(W_{t_j}^0\psi_j)<1 \quad\mbox{for all}\quad
j\in\{1,\dots,n\},
\end{equation}
then the operator $S$ is bounded on the Nakano space
$L^{p(\cdot)}(\Gamma,w)$ with weight $w$ given by
\begin{equation}\label{eq:weight}
w(\tau):=\prod_{j=1}^n \psi_j(\tau) \quad (\tau\in\Gamma).
\end{equation}

\item[{\rm (b)}]
If the operator $S$ is bounded on the Nakano space
$L^{p(\cdot)}(\Gamma,w)$ with weight $w$ given by
{\rm(\ref{eq:weight})}, then $\Gamma$ is a Carleson curve and
\[
0\le 1/p(t_j)+\alpha(W_{t_j}^0\psi_j),
\quad
1/p(t_j)+\beta(W_{t_j}^0\psi_j)\le 1
\quad\mbox{for all}\quad
j\in\{1,\dots,n\}.
\]

\item[{\rm (c)}]
If $\Gamma$ is a rectifiable Jordan curve and the operator $S$
is bounded on the Nakano space $L^{p(\cdot)}(\Gamma,w)$ with
weight $w$ given by {\rm(\ref{eq:weight})}, then $\Gamma$ is a
Carleson curve and conditions {\rm(\ref{eq:boundedness-condition})}
are fulfilled.
\end{enumerate}
\end{theorem}
A bounded linear operator on a Banach space $X$ is said to be Fredholm if its
image ${\rm Im}\,A$ is closed in $X$ and the numbers $\dim{\rm Ker}\,A$
and $\dim(X/{\rm Im}\,A)$ are finite.

We equip a rectifiable Jordan curve $\Gamma$ with the
counter-clockwise orientation. Without loss of generality we will
assume that the origin lies inside the domain bounded by $\Gamma$.
By $PC(\Gamma)$ we denote the set of all $a\in L^\infty(\Gamma)$
for which the one-sided limits
\[
a(t\pm 0):=\lim_{\tau\to t\pm 0}a(\tau)
\]
exist and are finite at each point $t\in\Gamma$; here $\tau\to
t-0$ means that $\tau$ approaches $t$ following the orientation of
$\Gamma$, while $\tau\to t+0$ means that $\tau$ goes to $t$ in the
opposite direction. Functions in $PC(\Gamma)$ are called
\textit{piecewise continuous} functions.

In 1968, Gohberg and Krupnik \cite[Theorem~4]{GK68} (see also
\cite[Chap.~9, Theorem~3.1]{GK92}) found criteria for
one-sided invertibility of one-dimensional singular integral
operators of the form
\[
A=aP+bQ,\quad\mbox{where}\quad a,b\in PC(\Gamma),\quad
P:=(I+S)/2,\quad Q:=(I-S)/2
\]
acting on standard Lebesgue spaces $L^p(\Gamma)$ over Lyapunov
curves. Their Fredholm theory for one-dimensional singular integral
operators on standard Lebesgue spaces $L^p(\Gamma,w)$ with
Khvedelidze weights (\ref {eq:Khvedelidze-weight}) over Lyapunov
curves is presented in the monograph \cite{GK92} first published
in Russian in 1973. Generalizations of this theory to the case of
arbitrary Muckenhoupt weights and arbitrary Carleson curves are
contained in the monograph by B\"ottcher and Yu. Karlovich
\cite{BK97}.

Fredholmness of one-dimensional singular integral operators on
Nakano spaces (variable Lebesgue spaces) over sufficiently smooth
curves was studied for the first time by Kokilashvili and S. Samko
\cite{KS03}. The closely related Riemann-Hilbert boundary value
problem in weighted classes of Cauchy type integrals with density
in $L^{p(\cdot)}(\Gamma)$ was considered by Kokilashvili,
Paatashvili, and S. Samko \cite{KP08,KP09,KPS05,KPS08}. The author
\cite{Karlovich09-IWOTA07} found a Fredholm criterion for an
arbitrary operator in the Banach algebra of one-dimensional
singular integral operators with piecewise continuous coefficients
acting on Nakano spaces $L^{p(\cdot)}(\Gamma,w)$ with radial
oscillating weights (\ref{eq:radial-weight}) over so-called
logarithmic Carleson curves. Roughly speaking, logarithmic
Carleson curves are Carleson curves $\Gamma$ for which the weight
$\eta_t(\tau)$ is equivalent to a power weight
$|\tau-t|^{\lambda_t}$ with some $\lambda_t\in\R$ for each
$t\in\Gamma$. Further, this technical assumption on Carleson
curves was removed in \cite{Karlovich10-IWOTA08} but only in the
non-weighted case.

The aim of this paper is to prove an analogue of the
Gohberg-Krupnik Fredholm criterion for the operator
$aP+bQ$ acting on $L^{p(\cdot)}(\Gamma,w)$ in the case of
arbitrary Carleson curves and a wide class of weights, in
particular, including radial oscillating weights
(\ref{eq:radial-weight}). Having this result at hands, one can
construct a Fredholm theory for the Banach algebra of singular
integral operators with piecewise continuous coefficients by using
the machinery developed in \cite{BK97} exactly in the same way as
it was done in \cite{Karlovich09-IWOTA07,Karlovich10-IWOTA08}. We will not
present these results in this paper.

Now we prepare the formulation of our main result. Let
$L^{p(\cdot)}(\Gamma,w)$ be as in Theorem~\ref{th:boundedness}.
One can show (see Section~\ref{subsection:Indicator_functions}) that the functions
$\alpha_t^*,\beta_t^*:\R\to\R$ given by
\begin{equation}\label{eq:indicator*-1}
\alpha_{t_j}^*(x):=\alpha(W_{t_j}^0(\eta_t^x\psi_j)), \quad
\beta_{t_j}^*(x):=\beta(W_{t_j}^0(\eta_t^x\psi_j))
\end{equation}
for $j\in\{1,\dots,n\}$ and by
\begin{equation}\label{eq:indicator*-2}
\alpha_t^*(x):=\alpha\big(W_t^0(\eta_t^x)\big), \quad
\beta_t^*(x):=\beta\big(W_t^0(\eta_t^x)\big)
\end{equation}
for $t\notin\Gamma\setminus\{t_1,\dots,t_n\}$ are well-defined and
the set
\[
Y(p(t),\alpha_t^*,\beta_t^*):= \big\{
\gamma=x+iy\in\C:1/p(t)+\alpha_t^*(x)\le y\le 1/p(t)+\beta_t^*(x)
\big\}
\]
is connected and contains points with arbitrary real parts. Given
$z_1,z_2\in\C$, let
\[
\cL(z_1,z_2;p(t),\alpha_t^*,\beta_t^*):=
\big\{M_{z_1,z_2}(e^{2\pi\gamma}):\gamma\in
Y(p(t),\alpha_t^*,\beta_t^*)\big\}\cup\{z_1,z_2\},
\]
where
\begin{equation}\label{eq:Moebius}
M_{z_1,z_2}(\zeta):=(z_2\zeta-z_1)/(\zeta-1)
\end{equation}
is the M\"obius transform. The set
$\cL(z_1,z_2;p(t),\alpha_t^*,\beta_t^*)$ is referred to as the
\textit{leaf} about (or between) $z_1$ and $z_2$ determined by
$p(t),\alpha_t^*,\beta_t^*$. This is a connected set containing $z_1$
and $z_2$ for every $t\in\Gamma$.

For $a\in PC(\Gamma)$, denote by $\cR(a)$ the essential range of
$a$, i.e. let $\cR(a)$ be the set
\[
\bigcup_{t\in\Gamma}\{a(t-0),a(t+0)\}.
\]
Let $J_a$ stand the set of all points at which $a$ has a jump.
Clearly, we may write
\[
\cR(a)=\bigcup_{t\in\Gamma\setminus J_a}\{a(t)\}\cup \bigcup_{t\in
J_a}\{a(t-0),a(t+0)\}.
\]
We will say that a function $a\in PC(\Gamma)$ is
$L^{p(\cdot)}(\Gamma,w)$-nonsingular if
\[
0\notin\cR(a)\cup\bigcup_{t\in J_a}\cL(a(t-0),a(t+0);p(t),\alpha_t^*,\beta_t^*).
\]

Our second main result reads as follows.
\begin{theorem}\label{th:Fredholmness}
Suppose $\Gamma$ is a Carleson Jordan curve, $a,b\in PC(\Gamma)$,
and $p:\Gamma\to(1,\infty)$ is a continuous function satisfying
the Dini-Lipschitz condition
{\rm(\ref{eq:Dini-Lipschitz})}. Let $t_1,\dots,t_n\in\Gamma$ be
pairwise distinct points and
$\psi_j:\Gamma\setminus\{t_j\}\to(0,\infty)$ be continuous
functions such that the functions $W_{t_j}\psi_j$ are regular and
conditions {\rm(\ref{eq:boundedness-condition})} are fulfilled for
all $j\in\{1,\dots,n\}$. The operator $aP+bQ$ is Fredholm
on the Nakano space $L^{p(\cdot)}(\Gamma,w)$ with
weight $w$ given by {\rm(\ref{eq:weight})} if and only if
$\inf\limits_{t\in\Gamma}|b(t)|>0$
and the function $a/b$ is $L^{p(\cdot)}(\Gamma,w)$-nonsingular.
\end{theorem}
For $b=1$, the above result generalizes \cite[Theorem~4.5]{Karlovich09-IWOTA07},
where the weights of the form (\ref{eq:radial-weight}) were considered
over so-called logarithmic Carleson curves, and \cite[Theorem~2.2]{Karlovich10-IWOTA08},
where underlying curves were arbitrary Carleson curves but no weights were
involved.

Although the main results of this paper are new, the methods of their proofs
are not new and known to experts in the field. We decided to provide
selfcontained proofs with complete formulations of auxiliary results
taken from other publications. So, this paper can be considered as a
short survey on the topic.

The paper is organized as follows. In Section~\ref{sec:Auxiliary_results},
we collect some auxiliary results on indices of submultiplicative
functions associated with curves and weights. In Section~\ref{sec:Maximal_operator},
following the approach of Kokilashvili, N. Samko, and S. Samko
\cite{KSS07-JFSA}, we prove that conditions (\ref{eq:boundedness-condition})
are sufficient for the boundedness of the maximal operator on Nakano spaces
$L^{p(\cdot)}(\Gamma,w)$ with weights of the form (\ref{eq:weight}).
With the aid of this result, we prove Theorem~\ref{th:boundedness}
in Section~\ref{sec:Singular_operator}. Section~\ref{sec:SIO} contains
basic results on singular integral operators with $L^\infty$ coefficients
on weighted Nakano spaces. Their proofs are analogous to the case of
standard weighted Lebesgue spaces (see e.g. \cite[Chap.~7-8]{GK92}).
In Section~\ref{sec:SIO-PC} we prove Theorem~\ref{th:Fredholmness}
following the approach of B\"ottcher and Yu.~Karlovich \cite[Chap.~7]{BK97}
(see also \cite{Karlovich09-IWOTA07,Karlovich10-IWOTA08}). Note that
Theorem~\ref{th:boundedness} plays a crucial role in the proof of
Theorem~\ref{th:Fredholmness}.
\section{Indices of submultiplicative functions}\label{sec:Auxiliary_results}
\subsection{Indices as limits}
The indices of a regular submultiplicative function
defined by (\ref{eq:indices-definition}) cen be calculated as limits
as $x\to 0$ and $x\to\infty$, respectively. The proof of the following
result can be found e.g. in \cite[Theorem~1.13]{BK97}.
\begin{theorem}[well-known]
\label{th:indices}
If a function $\varrho:(0,\infty)\to(0,\infty)$ is regular and
submultiplicative, then
\[
\alpha(\varrho)=\lim_{x\to 0}\frac{\log\varrho(x)}{\log x},
\quad
\beta(\varrho)=\lim_{x\to \infty}\frac{\log\varrho(x)}{\log x}
\]
and $-\infty<\alpha(\varrho)\le\beta(\varrho)<\infty$.
\end{theorem}
\subsection{Spirality indices}
The following result was proved in \cite[Theorem~1.18]{BK97} and
\cite[Proposition~3.1]{BK97}.
\begin{theorem}[B\"ottcher, Yu. Karlovich]
\label{th:BK-W-regularity}
Let $\Gamma$ be a simple Carleson curve and $t\in\Gamma$. For every
$x\in\R$, the functions $W_t(\eta_t^x)$ and $W_t^0(\eta_t^x)$ are regular
and submultiplicative and
\begin{eqnarray*}
\alpha\big(W_t(\eta_t^x)\big)
=
\alpha\big(W_t^0(\eta_t^x)\big)
&=&
\min\{\delta_t^-x,\delta_t^+x\},
\\
\beta\big(W_t(\eta_t^x)\big)
=
\beta\big(W_t^0(\eta_t^x)\big)
&=&
\max\{\delta_t^-x,\delta_t^+x\},
\end{eqnarray*}
where
\[
\delta_t^-:=\alpha(W_t^0\eta_t),\quad\delta_t^+:=\beta(W_t^0\eta_t).
\]
\end{theorem}
The numbers $\delta_t^-$ and $\delta_t^+$ are called the lower and upper spirality
indices of $\Gamma$ at $t$. If $\Gamma$ is locally smooth at $t$, then
$\delta_t^-=\delta_t^+=0$. One says that $\Gamma$ is a logarithmic Carleson
curve if
\[
\arg(\tau-t)=-\delta_t\log|\tau-t|+O(1)
\quad\mbox{as}\quad\tau\to t
\]
for every $t\in\Gamma$. It is not difficult to see, that for such Curves
$\delta_t^-=\delta_t^+=\delta_t$ for every $t\in\Gamma$. However, arbitrary
Carleson curves have much more complicated behavior. Indeed, for given numbers
$\alpha,\beta\in\R$ such that $\alpha\le\beta$, one can construct a Carleson
curve such that $\alpha=\delta_t^-$ and $\beta=\delta_t^+$ at some point
$t\in\Gamma$ (see \cite[Proposition~1.21]{BK97}).
\subsection{Indices of powerlikeness}
Fix $t\in\Gamma$. Let $w:\Gamma\to[0,\infty]$ be a weight on $\Gamma$ such
that $\log w\in L^1(\Gamma(t,R))$ for every $R\in(0,d_t]$. Put
\[
H_{w,t}(R_1,R_2):=
\frac{\displaystyle\exp\left(\frac{1}{|\Gamma(t,R_1)|}\int_{\Gamma(t,R_1)}\log w(\tau)|d\tau|\right)}
{\displaystyle\exp\left(\frac{1}{|\Gamma(t,R_2)|}\int_{\Gamma(t,R_2)}\log w(\tau)|d\tau|\right)},
\quad R_1,R_2\in(0,d_t].
\]
Following \cite[Section~3.2]{BK97}, we define
\[
(V_tw)(x):=\left\{\begin{array}{lll}
\displaystyle
\sup_{0<R\le d_t}H_{w,t}(xR,R) &\mbox{for}& x\in(0,1],
\\
\displaystyle
\sup_{0<R\le d_t}H_{w,t}(R,x^{-1}R) &\mbox{for}& x\in[1,\infty)
\end{array}\right.
\]
and
\[
(V_t^0w)(x):=\limsup_{R\to 0}H_{w,t}(xR,R)=\limsup_{R\to 0}H_{w,t}(R,x^{-1}R)
\]
for $x\in\R$.

A function $f:\Gamma\to[-\infty,\infty]$ is said to have bounded mean oscillation
at a point $t\in\Gamma$ if $f\in L^1(\Gamma)$ and
\[
\sup_{R>0}\frac{1}{|\Gamma(t,R)|}\int_{\Gamma(t,R)}|f(\tau)-\Delta_t(f,R)|\,|d\tau|<\infty,
\]
where
\[
\Delta_t(f,R):=\frac{1}{|\Gamma(t,R)|}\int_{\Gamma(t,R)}f(\tau)|d\tau|
\quad(R>0).
\]
The class of all functions of bounded mean oscillation at $t\in\Gamma$ is
denoted by $BMO(\Gamma,t)$.

The following result give sufficient conditions for the
regularity of $V_tw$ and $V_t^0w$. It was proved in \cite[Theorem~3.3(a)]{BK97}
and  \cite[Lemma~3.5(a)]{BK97}.
\begin{theorem}[B\"ottcher, Yu. Karlovich]
\label{th:BK-V-regularity}
Suppose $\Gamma$ is a simple Carleson curve and $t\in\Gamma$. If
$w:\Gamma\to[0,\infty]$ is a weight such that $\log w\in BMO(\Gamma,t)$,
then the functions $V_tw$ and $V_t^0w$ are regular and submultiplicative and
\[
\alpha(V_tw)=\alpha(V_t^0w),
\quad
\beta(V_tw)=\beta(V_t^0w).
\]
\end{theorem}
The numbers $\alpha(V_t^0w)$ and $\beta(V_t^0w)$ are called the lower and
upper indices of powerlikeness of $w$ at $t\in\Gamma$, respectively.
This terminology can be explained by the simple fact that for the power weight
$w(\tau)=|\tau-t|^\lambda$ its indices of powerlikeness coincide and are
equal to $\lambda$.
\begin{lemma}\label{le:V-relations}
Let $\Gamma$ be a simple Carleson curve and
$t_1,\dots,t_n\in\Gamma$ be pairwise distinct points. Suppose
$\psi_j:\Gamma\setminus\{t_j\}\to(0,\infty)$ are continuous
functions for $j\in\{1,\dots.n\}$ and $w$ is the weight given by
{\rm(\ref{eq:weight})}. If $V_{t_j}^0w$ is regular for some
$j\in\{1,\dots,n\}$, then $V_{t_j}^0\psi_j$ is also regular and
\[
\alpha(V_{t_j}^0w)=\alpha(V_{t_j}^0\psi_j),
\quad
\beta(V_{t_j}^0w)=\beta(V_{t_j}^0\psi_j).
\]
\end{lemma}
\begin{proof}
Without loss of generality, assume that $V_{t_1}^0w$ is regular. Suppose
$\Gamma_1\subset\Gamma$ is an arc that contains the point $t_1$ but does not
contain the points $t_2,\dots,t_n$. Assume that $\Gamma_1$ is homeomorphic
to a segment. Then the functions $\psi_2,\dots,\psi_n$ are continuous on the
compact set $\Gamma_1$. Therefore there exist constants $C_1$ and $C_2$ such
that
\[
0< C_1\le \psi_2(\tau)\dots\psi_n(\tau)\le C_2<+\infty
\quad\mbox{for all}\quad\tau\in\Gamma_1.
\]
Then $C_1\psi_1(\tau)\le w(\tau)\le C_2\psi_1(\tau)$ for all $\tau\in\Gamma_1$
and
\[
\frac{C_1}{C_2}H_{\psi_1,t_1}(R_1,R_2)\le H_{w,t_1}(R_1,R_2)\le\frac{C_2}{C_1}H_{\psi_1,t_1}(R_1,R_2)
\]
for all $R_1,R_2\in\Big(0,\max\limits_{\tau\in\Gamma_1}|\tau-t_1|\Big)$.
Thus,
\[
\frac{C_1}{C_2}(V_{t_1}^0\psi_1)(x)\le (V_{t_1}^0w)(x)\le\frac{C_2}{C_1}(V_{t_1}^0\psi_1)(x)
\quad\mbox{for all}\quad x\in\R.
\]
These inequalities imply that if $V_{t_1}^0w$ is regular, then $V_{t_1}^0\psi_1$
is also regular and their indices coincide: $\alpha(V_{t_1}^0\psi_1)=\alpha(V_{t_1}^0w)$
and $\beta(V_{t_1}^0\psi_1)=\beta(V_{t_1}^0w)$
\end{proof}
\subsection{Relations between indices of submultiplicative functions}
The following statement is proved by analogy with \cite[Proposition~3.1]{BK97}.
\begin{lemma}\label{le:indices-relations}
Let $\Gamma$ be a simple rectifiable curve and $t\in\Gamma$.
Suppose $\psi:\Gamma\setminus\{t\}\to(0,\infty)$ is a continuous
function and $W_t\psi$ is regular. Then, for every $s\in\R$, the
functions $W_t(\psi^s)$ and $W_t^0(\psi^s)$ are regular and
\begin{eqnarray*}
\alpha\big(W_t(\psi^s)\big)
=
\alpha\big(W_t^0(\psi^s)\big)
&=&
\left\{\begin{array}{lll}
s\alpha(W_t^0\psi) & \mbox{if} & s\ge 0,
\\
s\beta(W_t^0\psi) & \mbox{if} & s<0,
\end{array}\right.
\\
\beta\big(W_t(\psi^s)\big)
=
\beta\big(W_t^0(\psi^s)\big)
&=&
\left\{\begin{array}{lll}
s\beta(W_t^0\psi) & \mbox{if} & s\ge 0,
\\
s\alpha(W_t^0\psi) & \mbox{if} & s<0.
\end{array}\right.
\end{eqnarray*}
\end{lemma}
The next statement is certainly known to experts, however we were
unable to find the precise reference.
\begin{lemma}\label{le:WW-estimates}
Let $\Gamma$ be a simple rectifiable curve, $t\in\Gamma$, and
$\psi_1,\psi_2:\Gamma\setminus\{t\}\to (0,\infty)$ be
continuous functions such that the functions $W_t\psi_1$ and $W_t\psi_2$
are regular. Then the functions $W_t(\psi_1\psi_2)$ and
$W_t^0(\psi_1\psi_2)$ are regular and submultiplicative and
\[
\alpha(W_t\psi_1)+\alpha(W_t\psi_2)
\le
\alpha\big(W_t(\psi_1\psi_2)\big)
\le
\min\big\{
\alpha(W_t\psi_1)+\beta(W_t\psi_2),
\beta(W_t\psi_1)+\alpha(W_t\psi_2)
\big\},
\]
\[
\beta(W_t\psi_1)+\beta(W_t\psi_2)
\ge
\beta\big(W_t(\psi_1\psi_2)\big)
\ge
\max\big\{
\alpha(W_t\psi_1)+\beta(W_t\psi_2),
\beta(W_t\psi_1)+\alpha(W_t\psi_2)
\big\}.
\]
The same inequalities hold with $W_t$ replaced by $W_t^0$ in each occurrence.
\end{lemma}
\begin{proof}
Let $R\in(0,d_t]$ and $x\in(0,1]$. Then
\[
\frac{\displaystyle\max_{\tau\in\Gamma:|\tau-t|=xR}(\psi_1(\tau)\psi_2(\tau))}
{\displaystyle\min_{\tau\in\Gamma:|\tau-t|=R}(\psi_1(\tau)\psi_2(\tau))}
\le
\frac{\displaystyle\max_{\tau\in\Gamma:|\tau-t|=xR}\psi_1(\tau)}
{\displaystyle\min_{\tau\in\Gamma:|\tau-t|=R}\psi_1(\tau)}
\cdot
\frac{\displaystyle\max_{\tau\in\Gamma:|\tau-t|=xR}\psi_2(\tau)}
{\displaystyle\min_{\tau\in\Gamma:|\tau-t|=R}\psi_2(\tau)}.
\]
Taking the supremum over all $R\in(0,d_t]$, we obtain
\begin{equation}\label{eq:WW-estimates-1}
\big(W_t(\psi_1\psi_2)\big)(x)\le(W_t\psi_1)(x)(W_t\psi_2)(x)
\end{equation}
for all $x\in(0,1]$. Analogously it can be shown that this inequality
holds for $x\in(1,\infty)$. In particular, this implies that the
function $W_t(\psi_1\psi_2)$ is regular. Further, taking the logarithms
of both sides of (\ref{eq:WW-estimates-1}), dividing by $\log x$, and then
passing to the limits as $x\to 0$ and $x\to\infty$, we obtain
\begin{equation}\label{eq:WW-estimates-2}
\alpha(W_t\psi_1)+\alpha(W_t\psi_2)\le\alpha\big(W_t(\psi_1\psi_2)\big),
\quad
\beta\big(W_t(\psi_1\psi_2)\big)\le\beta(W_t\psi_1)+\beta(W_t\psi_2),
\end{equation}
respectively. Notice that the passage to the limits is justified by
Theorem~\ref{th:indices}.

Similarly,
\begin{eqnarray*}
&&
\frac{\displaystyle\max_{\tau\in\Gamma:|\tau-t|=xR}(\psi_1(\tau)\psi_2(\tau))}
{\displaystyle\min_{\tau\in\Gamma:|\tau-t|=R}(\psi_1(\tau)\psi_2(\tau))}
\ge
\frac{\displaystyle\min_{\tau\in\Gamma:|\tau-t|=xR}\psi_1(\tau)}
{\displaystyle\max_{\tau\in\Gamma:|\tau-t|=R}\psi_1(\tau)}
\cdot
\frac{\displaystyle\max_{\tau\in\Gamma:|\tau-t|=xR}\psi_2(\tau)}
{\displaystyle\min_{\tau\in\Gamma:|\tau-t|=R}\psi_2(\tau)}
\\
&&\ge
\left(\inf_{R\in(0,d_t]}
\frac{\displaystyle\min_{\tau\in\Gamma:|\tau-t|=xR}\psi_1(\tau)}
{\displaystyle\max_{\tau\in\Gamma:|\tau-t|=R}\psi_1(\tau)}
\right)
\cdot
\frac{\displaystyle\max_{\tau\in\Gamma:|\tau-t|=xR}\psi_2(\tau)}
{\displaystyle\min_{\tau\in\Gamma:|\tau-t|=R}\psi_2(\tau)}
\\
&&=
\left(\sup_{R\in(0,d_t]}
\frac{\displaystyle\max_{\tau\in\Gamma:|\tau-t|=R}\psi_1(\tau)}
{\displaystyle\min_{\tau\in\Gamma:|\tau-t|=xR}\psi_1(\tau)}
\right)^{-1}
\cdot
\frac{\displaystyle\max_{\tau\in\Gamma:|\tau-t|=xR}\psi_2(\tau)}
{\displaystyle\min_{\tau\in\Gamma:|\tau-t|=R}\psi_2(\tau)}
\\
&&=
\frac{1}{(W_t\psi_1)(x^{-1})}
\cdot
\frac{\displaystyle\max_{\tau\in\Gamma:|\tau-t|=xR}\psi_2(\tau)}
{\displaystyle\min_{\tau\in\Gamma:|\tau-t|=R}\psi_2(\tau)}.
\end{eqnarray*}
Taking the supremum over all $R\in(0,d_t]$, we obtain
\begin{equation}\label{eq:WW-estimates-3}
\big(W_t(\psi_1\psi_2)\big)(x)\ge\frac{(W_t\psi_2)(x)}{(W_t\psi_1)(x^{-1})}
\end{equation}
for $x\in(0,1]$. Then
\[
\frac{\log\big(W_t(\psi_1\psi_2)\big)(x)}{\log x}
\le
\frac{\log(W_t\psi_2)(x)}{\log x}
+
\frac{\log (W_t\psi_1)(x^{-1})}{\log(x^{-1})}.
\]
Passing to the limit as $x\to 0$, we obtain
\begin{equation}\label{eq:WW-estimates-4}
\alpha\big(W_t(\psi_1\psi_2)\big)\le\alpha(W_t\psi_2)+\beta(W_t\psi_1).
\end{equation}
In the same way it can be shown that (\ref{eq:WW-estimates-3}) is also fulfilled
for $x\in(1,\infty)$. This implies that
\begin{equation}\label{eq:WW-estimates-5}
\beta\big(W_t(\psi_1\psi_2)\big)\ge\beta(W_t\psi_2)+\alpha(W_t\psi_1).
\end{equation}
Replacing $\psi_1$ by $\psi_2$ and vice versa, we also get
\begin{equation}\label{eq:WW-estimates-6}
\alpha\big(W_t(\psi_1\psi_2)\big)\le\alpha(W_t\psi_1)+\beta(W_t\psi_2),
\quad
\beta\big(W_t(\psi_1\psi_2)\big)\ge\beta(W_t\psi_1)+\alpha(W_t\psi_2).
\end{equation}
Combining inequalities (\ref{eq:WW-estimates-2}) and
(\ref{eq:WW-estimates-4})--(\ref{eq:WW-estimates-6}) we arrive at
the statement of the lemma for $W_t$. The statement for $W_t^0$
follows from the statement for $W_t$ and \cite[Lemma~1.16]{BK97}.
\end{proof}
From \cite[Theorem~3.3(c)]{BK97} and \cite[Lemma~3.16]{BK97} we get the following.
\begin{lemma}\label{le:WV-relations}
Let $\Gamma$ be a simple Carleson curve and $t\in\Gamma$. If
$\psi:\Gamma\setminus\{t\}\to(0,\infty)$ is a continuous function
such that $W_t\psi$ is regular, then the functions $W_t^0\psi$ and
$V_t^0\psi$ are regular and submultiplicative and
\[
\alpha(W_t^0\psi)=\alpha(V_t^0\psi),
\quad
\beta(W_t^0\psi)=\beta(V_t^0\psi).
\]
\end{lemma}
The next statement is taken from \cite[Lemma~3.17]{BK97}.
\begin{lemma}\label{le:WV-estimates}
Let $\Gamma$ be a simple Carleson curve and $t\in\Gamma$. Suppose
$\psi:\Gamma\setminus\{t\}\to(0,\infty)$ is a continuous function
such that the function $W_t\psi$ is regular and
$w:\Gamma\to[0,\infty]$ is a weight such that $\log w\in
BMO(\Gamma,t)$. Then the function $V_t^0(\psi w)$ is regular and
submultiplicative and
\[
\alpha(V_t^0w)+\alpha(W_t\psi)
\le
\alpha\big(V_t^0(\psi w)\big)
\le
\min\big\{\alpha(V_t^0w)+\beta(W_t\psi),\beta(V_t^0w)+\alpha(W_t\psi)\big\},
\]
\[
\beta(V_t^0w)+\alpha(W_t\psi)
\ge
\beta\big(V_t^0(\psi w)\big)
\ge
\max\big\{\alpha(V_t^0w)+\beta(W_t\psi),\beta(V_t^0w)+\alpha(W_t\psi)\big\}.
\]
\end{lemma}
\subsection{Estimates of weights with one singularity by power
weights} Fix $t_0\in\Gamma$. Let $\omega(t_0,\delta)$ denote the
open arc on $\Gamma$ which contains $t_0$ and whose endpoints lie
on the circle
\[
\{\tau\in\C:|\tau-t_0|=\delta\}.
\]
It is clear that $\omega(t_0,\delta)\subset\Gamma(t_0,\delta)$,
however, it may happen that
$\omega(t_0,\delta)\ne\Gamma(t_0,\delta)$. The following lemma was
obtained in \cite[Lemma~3.2]{Karlovich10-MN}.
\begin{lemma}\label{le:estimate}
Let $\Gamma$ be a simple Carleson curve and $t_0\in\Gamma$.
Suppose $\psi:\Gamma\setminus\{t_0\}\to(0,\infty)$ is a continuous
function and $W_{t_0}\psi$ is regular. Let $\eps>0$ and $\delta$
be such that $0<\delta<d_{t_0}$. Then there exist positive
constants $C_j=C_j(\eps,\delta,w)$, where $j=1,2$, such that
\begin{equation}\label{eq:estimate-1}
\frac{\psi(t)}{\psi(\tau)} \le
C_1\left|\frac{t-t_0}{\tau-t_0}\right|^{\beta(W_{t_0}\psi)+\eps}
\end{equation}
for all $t\in\Gamma\setminus\omega(t_0,\delta)$ and all
$\tau\in\omega(t_0,\delta)$; and
\begin{equation}\label{eq:estimate-2}
\frac{\psi(t)}{\psi(\tau)} \le
C_2\left|\frac{t-t_0}{\tau-t_0}\right|^{\alpha(W_{t_0}\psi)-\eps}
\end{equation}
for all $t\in\omega(t_0,\delta)$ and all
$\tau\in\Gamma\setminus\omega(t_0,\delta)$.
\end{lemma}
\section{The boundedness of the maximal operator on weighted Nakano spaces}
\label{sec:Maximal_operator}
\subsection{Muckenhoupt weights on Carleson curves}
For a function $f\in L^1(\Gamma)$ the maximal function $Mf$ of $f$ on
$\Gamma$ is defined by
\[
(Mf)(t):=\sup_{R>0}\frac{1}{|\Gamma(t,R)|}\int_{\Gamma(t,R)}|f(\tau)|\,|d\tau|
\quad (t\in\Gamma).
\]
The map $M:f\mapsto Mf$ is referred to as the \textit{maximal
operator}.

The boundedness of the operators $M$ and $S$ on standard weighted
Lebesgue spaces is well understood (see e.g.
\cite{BK97,Dynkin87,GGKK98,KKP98,Stein93}).
\begin{theorem}[well-known]
Suppose $\Gamma$ is a simple Carleson curve. If $T$ is one of the
operators $M$ or $S$ and $1<p<\infty$, then $T$ is bounded on
$L^p(\Gamma,w)$ if and only if $w$ is a Muckenhoupt weight, $w\in
A_p(\Gamma)$, that is,
\[
\sup_{t\in\Gamma}\sup_{R>0}
\left(\frac{1}{R}\int_{\Gamma(t,R)}w^p(\tau)\,|d\tau|\right)^{1/p}
\left(\frac{1}{R}\int_{\Gamma(t,R)}w^{-q}(\tau)\,|d\tau|\right)^{1/q}<\infty
\]
where $1/p+1/q=1$.
\end{theorem}
We now consider weights $\psi$ which are continuous and nonzero on
$\Gamma$ minus a point $t$. If the function $W_t\psi$ is regular,
then its indices are well defined. The following theorem is due to
B\"ottcher and Yu. Karlovich \cite[Theorem~2.33]{BK97}. It
provides us with a very useful tool for checking the Muckenhoupt
condition once the indices of $W_t\psi$ are available.
\begin{theorem}[B\"ottcher, Yu.~Karlovich]
\label{th:BK}
Let $1<p<\infty$ and $\Gamma$ be a simple Carleson curve and
$t\in\Gamma$. Suppose $\psi:\Gamma\setminus\{t\}\to(0,\infty)$ is
a continuous function and $W_t\psi$ is regular. Then $\psi\in
A_p(\Gamma)$ if and only if
\[
0<1/p+\alpha(W_t^0\psi),\quad 1/p+\beta(W_t^0\psi)<1.
\]
\end{theorem}
\subsection{The boundedness of $M$ on Nakano spaces with Khvedelidze weights}
The proof of Theorem~\ref{th:KPS} in \cite{KPS06} is based on the following result
proved in \cite[Theorem~A]{KS08}. It will also play  an essential role in our proof
of Theorem~\ref{th:boundedness}.
\begin{theorem}[Kokilashvili, S.~Samko]\label{th:KS-Khvedelidze}
Suppose $\Gamma$ is a simple Carleson curve and
$p:\Gamma\to(1,\infty)$ is a continuous function satisfying the
Dini-Lipschitz condition {\rm(\ref{eq:Dini-Lipschitz})}. Let
$t_1,\dots,t_n\in\Gamma$ be pairwise distinct points and
$\lambda_1,\dots,\lambda_n\in\R$. The maximal operator $M$ is
bounded on the Nakano space $L^{p(\cdot)}(\Gamma,w)$ with the
Khvedelidze weight $w$ given by {\rm(\ref{eq:Khvedelidze-weight})}
if and only if $0<1/p(t_j)+\lambda_j<1$ for all
$j\in\{1,\dots,n\}$.
\end{theorem}
\subsection{Sufficient condition for the boundedness of $M$ involving Muckenhoupt weights}
\label{sec:Muckenhoupt-ersatz} Although a complete
characterization of weights for which $M$ is bounded on weighted
variable Lebesgue spaces is still unknown in the setting of
arbitrary Carleson curves (see \cite{HD08} for the setting of
$\R^n$), one of the most significant recent results to achieve
this aim is the following sufficient condition (see
\cite[Theorem~${\rm A}^\prime$]{KSS07-JFSA}).
\begin{theorem}[Kokilashvili, N.~Samko, S.~Samko]
\label{th:KSS} Let $\Gamma$ be a simple Carleson curve,
$p:\Gamma\to(1,\infty)$ be a continuous function satisfying the
Dini-Lipschitz condition {\rm(\ref{eq:Dini-Lipschitz})}, and
$w:\Gamma\to[0,\infty]$ be a weight such that $w^{p/p_*}$ belongs
to the Muckenhoupt class $A_{p_*}(\Gamma)$, where
\begin{equation}\label{eq:p-min}
p_*:=p_*(\Gamma):=\min_{\tau\in\Gamma}p(\tau).
\end{equation}
Then $M$ is bounded on $L^{p(\cdot)}(\Gamma,w)$.
\end{theorem}
This theorem does not contain the sufficiency portion of
Theorem~\ref{th:KS-Khvedelidze} whenever $p$ is variable because for the
weight $\varrho(\tau)=|\tau-t|^\lambda$ the condition
$\varrho^{p/p_*}\in A_{p_*}(\Gamma)$ is equivalent to
$-1/p(t)<\lambda<(p_*-1)/p(t)$, while the ``correct" interval for
$\lambda$ is wider:
\[
-1/p(t)<\lambda<(p(t)-1)/p(t).
\]
This means that the conditions of Theorem~\ref{th:KSS} cannot be
necessary unless $p$ is constant.
\subsection{Sufficient conditions for the boundedness of $M$ on weighted Nakano spaces}
Recall that two weights $w_1$ and $w_2$ on $\Gamma$ are said to be equivalent
if there is a bounded and bounded away from zero function $f$ on $\Gamma$
such that $w_1=fw_2$.

Now we will apply Theorem~\ref{th:KSS} to the weight $w$ given by
(\ref{eq:weight}).
\begin{lemma}\label{le:ersatz}
Let $\Gamma$ be a simple Carleson curve,  $p:\Gamma\to(1,\infty)$ be
a continuous function satisfying the Dini-Lipschitz condition
{\rm(\ref{eq:Dini-Lipschitz})}, and $t\in\Gamma$. Suppose
$\psi:\Gamma\setminus\{t\}\to(0,\infty)$ is a continuous functions such that
the function $W_t\psi$ is regular. If
\begin{equation}\label{eq:ersatz-1}
0<1/p(t)+\alpha(W_t^0\psi),
\quad
1/p(t)+\beta(W_t^0\psi)<p_*/p(t),
\end{equation}
where $p_*$ is defined by {\rm(\ref{eq:p-min})}, then the operator $M$ is
bounded on $L^{p(\cdot)}(\Gamma,\psi)$.
\end{lemma}
\begin{proof}
The proof is analogous to the proof of \cite[Lemma~2.2]{Karlovich10-MN}.
Taking into account Lemma~\ref{le:indices-relations}, we see that the function
$W_t(\psi^{p(t)/p_*})$ is regular and inequalities (\ref{eq:ersatz-1}) are
equivalent to
\begin{eqnarray*}
0<\frac{1}{p_*}+\frac{p(t)}{p_*}\alpha(W_t^0\psi)
&=&
\frac{1}{p_*}+\alpha\big(W_t^0(\psi^{p(t)/p_*})\big),
\\
1>\frac{1}{p_*}+\frac{p(t)}{p_*}\beta(W_t^0\psi)
&=&
\frac{1}{p_*}+\beta\big(W_t^0(\psi^{p(t)/p_*})\big).
\end{eqnarray*}
By Theorem~\ref{th:BK}, the latter inequalities are equivalent to
$\psi^{p(t)/p_*}\in A_{p_*}(\Gamma)$.

Let us show that the weights $\psi^{p/p_*}$ and $\psi^{p(t)/p_*}$ are equivalent.
Fix $\eps>0$. Since $\psi:\Gamma\setminus\{t\}\to(0,\infty)$ is continuous,
from Lemma~\ref{le:estimate} it follows that there exist
$C_1,C_2>0$ such that
\[
C_1|\tau-t|^{\beta(W_t^0\psi)+\eps}
\le
\psi(\tau)
\le
C_2|\tau-t|^{\alpha(W_t^0\psi)-\eps}
\]
for all $\tau\in\Gamma\setminus\{t\}$. Then
\begin{eqnarray}
\log C_1+\big(\beta(W_t^0\psi)+\eps\big)\log|\tau-t|
&\le& \log\psi(\tau),
\label{eq:ersatz-2}
\\
\log C_2+\big(\alpha(W_t^0\psi)-\eps\big)\log|\tau-t|
&\ge& \log\psi(\tau)
\label{eq:ersatz-3}
\end{eqnarray}
for all $\tau\in\Gamma\setminus\{t\}$. By the Dini-Lipschitz condition
(\ref{eq:Dini-Lipschitz}),
\begin{equation}\label{eq:ersatz-4}
-\frac{C_\Gamma}{-\log|\tau-t|}\le p(\tau)-p(t)\le\frac{C_\Gamma}{-\log|\tau-t|}
\end{equation}
for all $\tau\in\Gamma\setminus\{t\}$ such that $|\tau-t|\le 1/2$. Multiplying
inequalities (\ref{eq:ersatz-2})--(\ref{eq:ersatz-4}), we see that the
function
\[
F_t(\tau):=\frac{p(\tau)-p(t)}{p_*}\log\psi(\tau)
\]
is bounded on $\Gamma(t,1/2)\setminus\{t\}$. Obviously, it is also bounded on
$\Gamma\setminus\Gamma(t,1/2)$. Therefore
\[
\frac{\psi(\tau)^{p(\tau)/p_*}}{\psi(\tau)^{p(t)/p_*}}=\exp(F_t(\tau))
\]
is bounded and bounded away from zero on $\Gamma\setminus\{t\}$. Thus the weights
$\psi^{p/p_*}$ and $\psi^{p(t)/p_*}$ are equivalent. In particular, this
implies that $\psi^{p/p_*}\in A_{p_*}(\Gamma)$ if and only if
$\psi^{p(t)/p_*}\in A_{p_*}(\Gamma)$. Thus, inequalities (\ref{eq:ersatz-1})
imply that $\psi^{p/p_*}\in A_{p_*}(\Gamma)$. Applying Theorem~\ref{th:KSS},
we finally conclude that the maximal operator $M$ is bounded on
$L^{p(\cdot)}(\Gamma,\psi)$.
\end{proof}
\begin{theorem}\label{th:boundedness-maximal}
Suppose $\Gamma$ is a simple Carleson curve and
$p:\Gamma\to(1,\infty)$ is a continuous function satisfying the
Dini-Lipschitz condition {\rm(\ref{eq:Dini-Lipschitz})}. Let
$t_1,\dots,t_n\in\Gamma$ be pairwise distinct points and
$\psi_j:\Gamma\setminus\{t_j\}\to(0,\infty)$ be continuous
functions such that the functions $W_{t_j}\psi_j$ are regular for
all $j\in\{1,\dots,n\}$. If for all $j\in\{1,\dots,n\}$,
\begin{equation}\label{eq:boundedness-maximal-1}
0<1/p(t_j)+\alpha(W_{t_j}^0\psi_j), \quad
1/p(t_j)+\beta(W_{t_j}^0\psi_j)<1,
\end{equation}
then the maximal operator $M$ is bounded on the Nakano space
$L^{p(\cdot)}(\Gamma,w)$ with weight $w$ given by
{\rm(\ref{eq:weight})}.
\end{theorem}
\begin{proof}
The idea of the proof is borrowed from \cite[Theorem~B]{KSS07-JFSA} (see also
\cite[Theorem~1.4]{Karlovich10-MN}).

We start the proof with a kind of separation of singularities of the weight.
Let arcs $\Gamma_1,\dots,\Gamma_n\subset\Gamma$ form a partition of $\Gamma$,
that is, each two arcs $\Gamma_i$ and $\Gamma_k$ may have only endpoints
in common and $\Gamma_1\cup\dots\cup\Gamma_n=\Gamma$. We will assume that
each arc $\Gamma_j$ is homeomorphic to a segment. Suppose that this partition
has the following property: each point $t_j$ belongs to $\Gamma_j$ and all
other points in $\{t_1,\dots,t_n\}\setminus\{t_j\}$ do not belong to $\Gamma_j$.

Obviously, the function
\[
w/\psi_j:=\psi_1\dots\psi_{j-1}\widetilde{\psi_j}\psi_{j+1}\dots\psi_n,
\]
where $\widetilde{a}$ denotes that the term $a$ is absent,
is continuous on the closed set $\Gamma_j$. Therefore,
\[
\inf_{\tau\in\Gamma_j}\left(\frac{w}{\psi_j}\right)(\tau)=:c_j>0,
\quad
\sup_{\tau\in\Gamma_j}\left(\frac{w}{\psi_j}\right)(\tau)=:C_j<+\infty.
\]
Hence, for every $f\in L^{p(\cdot)}(\Gamma,w)$, we have
\begin{eqnarray*}
\|f\|_{L^{p(\cdot)}(\Gamma,w)}
&\le &
\sum_{j=1}^n \|f\chi_{\Gamma_j}\|_{L^{p(\cdot)}(\Gamma,w)}
=
\sum_{j=1}^n \left\|f\left(\frac{w}{\psi_j}\right)\psi_j\chi_{\Gamma_j}\right\|_{L^{p(\cdot)}(\Gamma)}
\\
&\le&
\sum_{j=1}^n C_j\|f|_{\Gamma_j}\|_{L^{p(\cdot)}(\Gamma_j,\psi_j|_{\Gamma_j})}
\end{eqnarray*}
and
\begin{eqnarray*}
\|f|_{\Gamma_j}\|_{L^{p(\cdot)}(\Gamma_j,\psi_j|_{\Gamma_j})}
&=&
\|(f\psi_j)|_{\Gamma_j}\|_{L^{p(\cdot)}(\Gamma_j)}
=
\left\|(fw)|_{\Gamma_j}\left(\frac{w}{\psi_j}\right)^{-1}\Big|_{\Gamma_j}\right\|_{L^{p(\cdot)}(\Gamma_j)}
\\
&\le&
\frac{1}{c_j}\|(fw)|_{\Gamma_j}\|_{L^{p(\cdot)}(\Gamma_j)}
=
\frac{1}{c_j}\|fw\chi_{\Gamma_j}\|_{L^{p(\cdot)}(\Gamma)}
\\
&\le&
\frac{1}{c_j}\|fw\|_{L^{p(\cdot)}(\Gamma)}
=
\frac{1}{c_j}\|f\|_{L^{p(\cdot)}(\Gamma,w)}
\end{eqnarray*}
for every $j\in\{1,\dots,n\}$. From these estimates it follows that it is sufficient
to prove that $M$ is bounded on $L^{p(\cdot)}(\Gamma_j,\psi_j|_{\Gamma_j})$ for each
$j\in\{1,\dots,n\}$.

Fix $j\in\{1,\dots,n\}$. For simplicity of notation, assume that $\Gamma_j=\Gamma$.
This does not cause any problem because
\[
\big(W_{t_j}(\psi_j|_{\Gamma_j})\big)(x)
\le
(W_{t_j}\psi_j)(x),
\quad
\big(W_{t_j}^0(\psi_j|_{\Gamma_j})\big)(x)
=
(W_{t_j}^0\psi_j)(x)
\]
for all $x\in\R$. Therefore, $W_{t_j}(\psi_j|_{\Gamma_j})$ and
$W_{t_j}^0(\psi_j|_{\Gamma_j})$ are regular and
\[
\alpha
:=
\alpha\big(W_{t_j}^0(\psi_j|_{\Gamma_j})\big)
=
\alpha(W_{t_j}^0\psi_j),
\quad
\beta
:=
\beta\big(W_{t_j}^0(\psi_j|_{\Gamma_j})\big)
=
\beta(W_{t_j}^0\psi_j).
\]

It is easily seen that $M$ is bounded on $L^{p(\cdot)}(\Gamma,\psi_j)$ if and
only if the operator
\[
(M_jf)(t):=\sup_{R>0}\frac{\psi_j(t)}{|\Gamma(t,R)}\int_{\Gamma(t,R)}\frac{|f(\tau)|}{\psi_j(\tau)}|d\tau|
\quad
(t\in\Gamma)
\]
is bounded on $L^{p(\cdot)}(\Gamma)$. From (\ref{eq:boundedness-maximal-1})
it follows that there is a small $\eps>0$ such that
\begin{equation}\label{eq:boundedness-maximal-2}
0<1/p(t_j)+\alpha-\eps\le 1/p(t_j)+\beta+\eps<1.
\end{equation}
Since $p:\Gamma\to(1,\infty)$ is continuous and $1/p(t_j)+\beta<1$, we can choose
a number $\delta\in(0,d_{t_j})$ such that the arc $\omega(t_j,\delta)$, which contains
$t_j$ and has the endpoints on the circle $\{\tau\in\C:|\tau-t_j|=\delta\}$, is so small
that $1+\beta p(t_j)<p_*$, where
\[
p_*:=p_*(\omega(t_j,\delta))=\min_{\tau\in\overline{\omega(t_j,\delta)}}p(\tau).
\]
Hence
\begin{equation}\label{eq:boundedness-maximal-3}
0<1/p(t_j)+\alpha\le 1/p(t_j)+\beta<p_*/p(t_j).
\end{equation}

For $f\in L^{p(\cdot)}(\Gamma)$, we have
\begin{eqnarray}
\label{eq:boundedness-maximal-4}
M_jf &\le& \chi_{\omega(t_j,\delta)} M_j\chi_{\omega(t_j,\delta)}f+
\chi_{\Gamma\setminus\omega(t_j,\delta)}M_j\chi_{\omega(t_j,\delta)}f
\\
\nonumber
&&+
\chi_{\omega(t_j,\delta)}M_j\chi_{\Gamma\setminus\omega(t_j,\delta)}f+
\chi_{\Gamma\setminus\omega(t_j,\delta)}M_j\chi_{\Gamma\setminus\omega(t_j,\delta)}f.
\end{eqnarray}

From (\ref{eq:boundedness-maximal-3}) and Lemma~\ref{le:ersatz} it follows that
$M_j$ is bounded on $L^{p(\cdot)}(\omega(t_j,\delta))$. Consequently, the
operator $\chi_{\omega(t_j,\delta)}M_j\chi_{\omega(t_j,\delta)}I$ is bounded
on $L^{p(\cdot)}(\Gamma)$.

For $\lambda\in\R$, by $M_j^\lambda$ denote the weighted maximal operator
defined by
\[
(M_j^\lambda f)(t):=\sup_{R>0}\frac{|t-t_j|^\lambda}{|\Gamma(t,R)|}
\int_{\Gamma(t,R)}\frac{|f(\tau)|}{|\tau-t_j|^\lambda}|d\tau|.
\]
From Lemma~\ref{le:estimate} it follows that
\begin{equation}\label{eq:boundedness-maximal-5}
\chi_{\Gamma\setminus\omega(t_j,\delta)}M_j\chi_{\omega(t_j,\delta)}f
\le
C_1 \chi_{\Gamma\setminus\omega(t_j,\delta)}M_j^{\beta+\eps}\chi_{\omega(t_j,\delta)}f
\le
C_1 M_j^{\beta+\eps}f
\end{equation}
and
\begin{equation}\label{eq:boundedness-maximal-6}
\chi_{\omega(t_j,\delta)}M_j\chi_{\Gamma\setminus\omega(t_j,\delta)}f
\le
C_2 \chi_{\omega(t_j,\delta)}M_j^{\alpha-\eps}\chi_{\Gamma\setminus\omega(t_j,\delta)}f
\le
C_2 M_j^{\alpha-\eps}f,
\end{equation}
where $C_1$ and $C_2$ are positive constants depending only on $\eps,\delta$,
and $\psi_j$. From (\ref{eq:boundedness-maximal-2}) and Theorem~\ref{th:KS-Khvedelidze}
it follows that the operators $M_j^{\alpha-\eps}$ and $M_j^{\beta+\eps}$ are
bounded on $L^{p(\cdot)}(\Gamma)$. From here and
(\ref{eq:boundedness-maximal-5})--(\ref{eq:boundedness-maximal-6}) we conclude that
$\chi_{\Gamma\setminus\omega(t_j,\delta)}M_j\chi_{\omega(t_j,\delta)}I$ and
$\chi_{\omega(t_j,\delta)}M_j\chi_{\Gamma\setminus\omega(t_j,\delta)}I$ are
bounded on $L^{p(\cdot)}(\Gamma)$.

Finally, since $\Gamma\setminus\omega(t_j,\delta)$ does not contain the
singularity of the weight $\psi_j$ (which is continuous on $\Gamma\setminus\{t_j\}$),
there exists a constant $C_3>0$ such that
\[
\chi_{\Gamma\setminus\omega(t_j,\delta)}M_j\chi_{\Gamma\setminus\omega(t_j,\delta)}f\le C_3Mf.
\]
Theorem~\ref{th:KS-Khvedelidze} and the above estimate yield the boundedness
of $\chi_{\Gamma\setminus\omega(t_j,\delta)}M_j\chi_{\Gamma\setminus\omega(t_j,\delta)}I$
on $L^{p(\cdot)}(\Gamma)$. Thus, all operators on the right-hand side of
(\ref{eq:boundedness-maximal-4}) are bounded on $L^{p(\cdot)}(\Gamma)$.
Therefore, the operator on the left-hand side of (\ref{eq:boundedness-maximal-4})
is bounded, too. This completes the proof of the boundedness of $M$
on $L^{p(\cdot)}(\Gamma_j,\psi_j|_{\Gamma_j})$.
\end{proof}
\section{The Cauchy singular integral operator on weighted Nakano spaces}
\label{sec:Singular_operator}
\subsection{Necessary conditions for the boundedness of the operator $S$}
We will need the following necessary condition for the boundedness of
$S$ on weighted Nakano spaces.
\begin{theorem}\label{th:Karlovich-necessity}
Let $\Gamma$ be a simple rectifiable curve and let $p:\Gamma\to(1,\infty)$ be
a continuous function satisfying the Dini-Lipschitz condition
{\rm(\ref{eq:Dini-Lipschitz})}. If $w:\Gamma\to[0,\infty]$ is an arbitrary
weight such that the operator $S$ is bounded on $L^{p(\cdot)}(\Gamma,w)$,
then $\Gamma$ is a Carleson curve, $\log w\in BMO(\Gamma,t)$, the functions
$V_tw$ and $V_t^0w$ are regular and submultiplicative, and
\begin{equation}\label{eq:Karlovich-necessity-1}
0\le 1/p(t)+\alpha(V_t^0w),\quad 1/p(t)+\beta(V_t^0w)\le 1
\end{equation}
for every $t\in\Gamma$. If, in addition, $\Gamma$ is a rectifiable Jordan curve,
then
\begin{equation}\label{eq:Karlovich-necessity-2}
0<1/p(t)+\alpha(V_t^0w),\quad 1/p(t)+\beta(V_t^0w)<1
\end{equation}
for every $t\in\Gamma$.
\end{theorem}
\begin{proof}
For simple curves, the statement follows from \cite[Lemma~4.9]{Karlovich03}
and \cite[Theorems~5.9 and~6.1]{Karlovich03}. For Jordan curves, inequality
(\ref{eq:Karlovich-necessity-2}) was proved in \cite[Corollary~4.2]{Karlovich09-IWOTA07}.
\end{proof}
\subsection{The boundedness of $M$ implies the boundedness of $S$}
One of the main ingredients of the proof of
Theorem~\ref{th:boundedness} is the following recent result by
Kokilashvili and S.~Samko \cite[Theorem~4.21]{KS09}.
\begin{theorem}[Kokilashvili, S. Samko]
\label{th:KS} Let $\Gamma$ be a simple Carleson curve. Suppose
that $p:\Gamma\to(1,\infty)$ is a continuous function satisfying
the Dini-Lipschitz condition {\rm(\ref{eq:Dini-Lipschitz})} and
$w:\Gamma\to[1,\infty]$ is a weight. If there exists a number
$p_0$ such that
\[
1<p_0<\min\limits_{\tau\in\Gamma}p(\tau)
\]
and $M$ is bounded on
$L^{p(\cdot)/(p(\cdot)-p_0)}(\Gamma,w^{-p_0})$, then $S$ is
bounded on $L^{p(\cdot)}(\Gamma,w)$.
\end{theorem}
\subsection{Proof of Theorem~{\rm\ref{th:boundedness}}}
\begin{proof}
(a) This part is proved by analogy with \cite[Theorem~2.1]{Karlovich10-IWOTA08}.
Since the function $p:\Gamma\to(1,\infty)$ is continuous and $\Gamma$ is compact, we
deduce that $\min\limits_{\tau\in\Gamma}p(\tau)>1$. If the inequalities
\[
1/p(t_j)+\beta(W_{t_j}^0\psi_j)<1,\quad j\in\{1,\dots,n\}
\]
are fulfilled, then there exists a number $p_0$ such that
\[
1<p_0<\min_{\tau\in\Gamma}p(\tau)
\]
and
\[
1/p(t_j)+\beta(W_{t_j}^0\psi_j)<1/p_0,\quad j\in\{1,\dots,n\}.
\]
Taking into account Lemma~\ref{le:indices-relations}, we see that
the functions $W_{t_j}(\psi_j^{-p_0})$ are regular and
the latter inequalities are equivalent to
\begin{equation}\label{eq:boundedness-proof-1}
0<1-\frac{p_0}{p(t_j)}-p_0\beta(W_{t_j}^0\psi_j)=
\frac{p(t_j)-p_0}{p(t_j)}+\alpha\big(W_{t_j}^0(\psi_{t_j}^{-p_0})\big),
\quad j\in\{1,\dots,n\}.
\end{equation}
Analogously, the inequalities
\[
0<1/p(t_j)+\alpha(W_{t_j}^0\psi_j),\quad j\in\{1,\dots,n\}
\]
are equivalent to
\begin{equation}\label{eq:boundedness-proof-2}
1>1-\frac{p_0}{p(t_j)}-p_0\alpha(W_{t_j}^0\psi_j)=
\frac{p(t_j)-p_0}{p(t_j)}+\beta\big(W_{t_j}^0(\psi_j^{-p_0})\big),
\quad
j\in\{1,\dots,n\}.
\end{equation}
From inequalities (\ref{eq:boundedness-proof-1})--(\ref{eq:boundedness-proof-2})
and Theorem~\ref{th:boundedness-maximal} it follows that the maximal operator
$M$ is bounded on $L^{p(\cdot)/(p(\cdot)-p_0)}(\Gamma,w^{-p_0})$. To
finish the proof of part (a), it remains to apply Theorem~\ref{th:KS}.

(b) If the operator $S$ is bounded on $L^{p(\cdot)}(\Gamma,w)$,
then from (\ref{eq:Karlovich-necessity-1}) it follows that
\[
0\le 1/p(t_j)+\alpha(V_{t_j}^0w),
\quad
1/p(t_j)+\beta(V_{t_j}^0w)\le 1
\quad\mbox{for all}\quad j\in\{1,\dots,n\}.
\]
Then, by Lemma~\ref{le:V-relations},
\[
0\le 1/p(t_j)+\alpha(V_{t_j}^0\psi_j),
\quad
1/p(t_j)+\beta(V_{t_j}^0\psi_j)\le 1
\quad\mbox{for all}\quad j\in\{1,\dots,n\}.
\]
Applying Lemma~\ref{le:WV-relations} to the above inequalities, we see that
\[
0\le 1/p(t_j)+\alpha(W_{t_j}^0\psi_j),
\quad
1/p(t_j)+\beta(W_{t_j}^0\psi_j)\le 1
\quad\mbox{for all}\quad j\in\{1,\dots,n\}.
\]
Part (b) is proved. The proof of part (c) follows the same lines with inequalities
(\ref{eq:Karlovich-necessity-2}) in place of (\ref{eq:Karlovich-necessity-1}).
\end{proof}
\section{Singular integral operators with $L^\infty$ coefficients}
\label{sec:SIO}
\subsection{Necessary conditions for Fredholmness}
In this section we will suppose that $\Gamma$ is a Carleson Jordan curve,
$p:\Gamma\to(1,\infty)$ is a continuous function, and $w:\Gamma\to[0,\infty]$
is an arbitrary weight (not necessarily of the form (\ref{eq:weight}))
such that $S$ is bounded on $L^{p(\cdot)}(\Gamma,w)$. Under these assumptions,
\[
P:=(I+S)/2,
\quad
Q:=(I-S)/2
\]
are bounded projections on $L^{p(\cdot)}(\Gamma,w)$
(see \cite[Lemma~6.4]{Karlovich03}). The operators of the form $aP+bQ$, where
$a,b \in L^\infty(\Gamma)$, are called \textit{singular integral operators}
(SIOs).
\begin{theorem}\label{th:necessity}
Suppose $a,b\in L^\infty(\Gamma)$. If $aP+bQ$ is Fredholm on $L^{p(\cdot)}(\Gamma,w)$,
then $a^{-1},b^{-1}\in L^\infty(\Gamma)$.
\end{theorem}
This result can be proved in the same way as \cite[Theorem~5.4]{Karlovich07}
where the case of Khvedelidze weights (\ref{eq:Khvedelidze-weight}) was considered.
\subsection{The local principle}
Two functions $a,b\in L^\infty(\Gamma)$ are said to be locally
equivalent at a point $t\in\Gamma$ if
\[
\inf\big\{\|(a-b)c\|_\infty\ :\ c\in C(\Gamma),\ c(t)=1\big\}=0.
\]
\begin{theorem}\label{th:local_principle}
Suppose  $a\in L^\infty(\Gamma)$ and for each $t\in\Gamma$ there
exists a function $a_t\in L^\infty(\Gamma)$ which is locally
equivalent to $a$ at $t$. If the operators $a_tP+Q$ are Fredholm
on $L^{p(\cdot)}(\Gamma,w)$ for all $t\in\Gamma$, then $aP+Q$ is
Fredholm on $L^{p(\cdot)}(\Gamma,w)$.
\end{theorem}
For weighted Lebesgue spaces this theorem is known as Simonenko's
local principle \cite{Simonenko65}. It follows from
\cite[Theorem~6.13]{Karlovich03}.
\subsection{Wiener-Hopf factorization}
The curve $\Gamma$ divides the complex plane $\mathbb{C}$ into the bounded
simply connected domain $D^+$ and the unbounded domain $D^-$. Recall that
without loss of generality we assumed that $0\in D^+$.
We say that a function
$a\in L^\infty(\Gamma)$ admits a \textit{Wiener-Hopf factorization on}
$L^{p(\cdot)}(\Gamma,w)$ if $a^{-1}\in L^\infty(\Gamma)$ and $a$ can be written
in the form
\begin{equation}\label{eq:WH}
a(t)=a_-(t)t^\kappa a_+(t)
\quad\mbox{a.e. on}\ \Gamma,
\end{equation}
where $\kappa\in\Z$, and the factors $a_\pm$ enjoy the following properties:
\[
\begin{array}{lll}
({\rm i}) &
a_-\in QL^{p(\cdot)}(\Gamma,w)\stackrel{\cdot}{+}\mathbb{C}, &
a_-^{-1}\in QL^{q(\cdot)}(\Gamma,1/w)\stackrel{\cdot}{+}\mathbb{C},
\\[2mm]
&a_+\in PL^{q(\cdot)}(\Gamma,1/w), &
a_+^{-1}\in PL^{p(\cdot)}(\Gamma,w),
\\[2mm]
({\rm ii}) & \mbox{the operator $a_+^{-1}Sa_+I$} &
\mbox{is bounded on $L^{p(\cdot)}(\Gamma,w)$},
\end{array}
\]
where $1/p(t)+1/q(t)=1$ for all $t\in\Gamma$.
One can prove that the number $\kappa$ is uniquely determined.
\begin{theorem}\label{th:factorization}
A function $a\in L^\infty(\Gamma)$ admits a Wiener-Hopf factorization
{\rm(\ref{eq:WH})} on $L^{p(\cdot)}(\Gamma,w)$ if and only if the operator $aP+Q$
is Fredholm on $L^{p(\cdot)}(\Gamma,w)$.
\end{theorem}
This theorem goes back to Simonenko \cite{Simonenko64,Simonenko68}.
For more about this topic we refer to \cite[Section~6.12]{BK97},
\cite[Section~5.5]{BS06}, \cite[Section~8.3]{GK92}
in the case of weighted Lebesgue spaces. Theorem~\ref{th:factorization} follows
from \cite[Theorem~6.14]{Karlovich03}.
\section{Singular integral operators with $PC$ coefficients}
\label{sec:SIO-PC}
\subsection{Indicator functions}\label{subsection:Indicator_functions}
Combining Theorems~\ref{th:BK-W-regularity} and \ref{th:Karlovich-necessity}
with Lemma~\ref{le:WV-relations}, we arrive at the following.
\begin{lemma}\label{le:indicator-existence}
Let $\Gamma$ be a Carleson Jordan curve, $p:\Gamma\to(1,\infty)$
be a continuous function satisfying the Dini-Lipschitz condition
{\rm(\ref{eq:Dini-Lipschitz})}, and $w:\Gamma\to[0,\infty]$ be
a weight such that the operator $S$ is bounded on the weighted Nakano space
$L^{p(\cdot)}(\Gamma,w)$. Then, for every $x\in\R$ and every $t\in\Gamma$,
the function $V_t^0(\eta_t^xw)$ is regular and submultiplicative.
\end{lemma}
The above lemma says that the functions
\[
\alpha_t(x):=\alpha\big(V_t^0(\eta_t^xw)\big),
\quad
\beta_t(x):=\beta\big(V_t^0(\eta_t^x w)\big)
\quad
(x\in\R)
\]
are well-defined
for every $t\in\Gamma$.
The shape of these functions can be described with the aid
of the following theorem.
\begin{theorem}\label{th:indicator-shape}
Let $\Gamma$ be a Carleson Jordan curve, $p:\Gamma\to(1,\infty)$ be a continuous
function satisfying the Dini-Lipschitz condition {\rm(\ref{eq:Dini-Lipschitz})},
$w:\Gamma\to[0,\infty]$ be a weight such that the operator $S$ is bounded on
the weighted Nakano space $L^{p(\cdot)}(\Gamma,w)$, and $t\in\Gamma$.
Then the functions $\alpha_t$ and $\beta_t$ enjoy the following
properties:
\begin{enumerate}
\item[{\rm(a)}]
$-\infty<\alpha_t(x)\le\beta_t(x)<+\infty$ for all $x\in\R$;

\item[{\rm(b)}]
$0<1/p(t)+\alpha_t(0)\le 1/p(t)+\beta_t(0)<1$;

\item[{\rm(c)}]
$\alpha_t$ is concave and $\beta_t$ is convex;

\item[{\rm(d)}]
$\alpha_t(x)$ and $\beta_t(x)$ have asymptotes as $x\to\pm\infty$ and the convex
regions
\[
\big\{x+iy\in\C:y<\alpha_t(x)\big\}
\quad\mbox{and}\quad
\big\{x+iy\in\C:y>\beta_t(x)\big\}
\]
may be separated by parallels to each of these asymptotes; to be more precise,
there exist real numbers $\mu_t^-,\mu_t^+,\nu_t^-,\nu_t^+$ such that
\[
0<1/p(t)+\mu_t^-\le 1/p(t)+\nu_t^-<1,
\quad
0<1/p(t)+\mu_t^+\le 1/p(t)+\nu_t^+<1
\]
and
\begin{eqnarray*}
\beta_t(x)=\nu_t^++\delta_t^+x+o(1)
&\mbox{as}& x\to+\infty,
\\
\beta_t(x)=\nu_t^-+\delta_t^-x+o(1)
&\mbox{as}& x\to-\infty,
\\
\alpha_t(x)=\mu_t^++\delta_t^+x+o(1)
&\mbox{as}& x\to+\infty,
\\
\alpha_t(x)=\mu_t^-+\delta_t^-x+o(1)
&\mbox{as}& x\to-\infty.
\end{eqnarray*}
\end{enumerate}
\end{theorem}
\begin{proof}
Part (a) follows from Lemma~\ref{le:indicator-existence}.
Theorem~\ref{th:Karlovich-necessity} yields part (b). Part (c)
is proved in \cite[Proposition~3.20]{BK97} under the assumption that
$p$ is constant and $w\in A_p(\Gamma)$. In our case the proof is literally
the same. Again, part (d) is proved in \cite[Theorem~3.31]{BK97}
for $w\in A_p(\Gamma)$ and constant $p$. This proof works equally in our
case because in view of Theorem~\ref{th:Karlovich-necessity} we can
apply Lemma~\ref{le:WV-estimates} under the assumption that the operator
$S$ is bounded on $L^{p(\cdot)}(\Gamma,w)$.
\end{proof}
From Theorem~\ref{th:indicator-shape}(a),(c) we immediately deduce that the
set
\[
Y(p(t),\alpha_t,\beta_t):=\big\{\gamma=x+iy\in\C:1/p(t)+\alpha_t(x)\le y\le 1/p(t)+\beta_t(x)\big\}
\]
is a connected set containing points with arbitrary real parts. Hence the
set
\[
\big\{e^{2\pi\gamma}:\gamma\in Y(p(t),\alpha_t,\beta_t)\big\}
\]
is connected
and contains points arbitrarily close to the origin and to the infinity.
Let $z_1,z_2\in\C$.
The M\"obius transform $M_{z_1,z_2}$ given by (\ref{eq:Moebius}) maps
$0$ and $\infty$ to $z_1$ and $z_2$, respectively. Thus, the leaf
$\cL(z_1,z_2;p(t),\alpha_t,\beta_t)$ is a connected set containing $z_1$ and
$z_2$. More information about leaves (with many examples and computer plots)
can be found in \cite[Chap.~7]{BK97}.
\begin{lemma}
Let $\Gamma$ be a Carleson Jordan curve, $p:\Gamma\to(1,\infty)$ be a continuous
function satisfying the Dini-Lipschitz condition {\rm(\ref{eq:Dini-Lipschitz})},
and $t_1,\dots,t_n\in\Gamma$ be pairwise distinct points. Suppose
$\psi_j:\Gamma\setminus\{t_j\}\to(0,\infty)$
are continuous functions such that the functions $W_{t_j}\psi_j$ are regular,
conditions {\rm(\ref{eq:boundedness-condition})} are fulfilled, and the weight
$w$  is given by {\rm(\ref{eq:weight})}.
\begin{enumerate}
\item[{\rm(a)}]
If $t\in\Gamma\setminus\{t_1,\dots,t_n\}$, then for every $x\in\R$, the functions
$W_t^0(\eta_t^x)$ and $V_t^0(\eta_t^x w)$ are regular and submultiplicative and
\begin{eqnarray*}
&&
\alpha\big(W_t^0(\eta_t^x)\big)
=
\alpha\big(V_t^0(\eta_t^x w)\big)
=
\min\{\delta_t^-x,\delta_t^+x\},
\\
&&
\alpha\big(W_t^0(\eta_t^x)\big)
=
\alpha\big(V_t^0(\eta_t^x w)\big)
=
\min\{\delta_t^-x,\delta_t^+x\}.
\end{eqnarray*}

\item[{\rm(b)}]
If $j\in\{1,\dots,n\}$, then for every $x\in\R$, the functions $W_{t_j}^0(\eta_{t_j}^x\psi_j)$
and $V_{t_j}^0(\eta_{t_j}^xw)$ are regular and submultiplicative and
\begin{equation}\label{eq:indicator-1}
\alpha\big(W_{t_j}^0(\eta_{t_j}^x \psi_j)\big)
=
\alpha\big(V_t^0(\eta_{t_j}^x w)\big),
\quad
\beta\big(W_{t_j}^0(\eta_{t_j}^x \psi_j)\big)
=
\beta\big(V_t^0(\eta_{t_j}^x w)\big).
\end{equation}
\end{enumerate}
\end{lemma}
\begin{proof}
Let us prove a slightly more difficult part (b). Fix $t_j\in\{t_1,\dots,t_n\}$
and $x\in\R$.
By Theorem~\ref{th:BK-W-regularity}, the functions $W_{t_j}(\eta_{t_j}^x)$
and $W_{t_j}^0(\eta_{t_j}^x)$ are regular and submultiplicative. Then, in
view of Lemma~\ref{le:WW-estimates}, the functions $W_{t_j}(\eta_{t_j}^x\psi_j)$
and $W_{t_j}^0(\eta_{t_j}^x\psi_j)$ are regular and submultiplicative.
By Lemma~\ref{le:WV-relations}, the function $V_{t_j}^0(\eta_{t_j}^x\psi_j)$
is regular and submultiplicative and
\begin{equation}\label{eq:idicator-2}
\alpha\big(W_{t_j}^0(\eta_{t_j}^x\psi_j)\big)
=
\alpha\big(V_{t_j}^0(\eta_{t_j}^x\psi_j)\big),
\quad
\beta\big(W_{t_j}^0(\eta_{t_j}^x\psi_j)\big)
=
\beta\big(V_{t_j}^0(\eta_{t_j}^x\psi_j)\big).
\end{equation}
From Theorem~\ref{th:Karlovich-necessity} we know that $\log w\in BMO(\Gamma,t)$.
Therefore, in view of Lemma~\ref{le:WV-estimates}, the function $V_{t_j}^0(\eta_{t_j}^xw)$
is regular and submultiplicative. By Lemma~\ref{le:V-relations},
\begin{equation}\label{eq:indicator-3}
\alpha\big(V_{t_j}^0(\eta_{t_j}^xw)\big)
=
\alpha\big(V_{t_j}^0(\eta_{t_j}^x\psi_j)\big),
\quad
\beta\big(V_{t_j}^0(\eta_{t_j}^xw)\big)
=
\beta\big(V_{t_j}^0(\eta_{t_j}^x\psi_j)\big).
\end{equation}
Combining (\ref{eq:idicator-2})--(\ref{eq:indicator-3}), we arrive at
(\ref{eq:indicator-1}). Part (b) is proved. The proof of part (a)
is analogous.
\end{proof}
This lemma says that, under the assumptions of Theorem~\ref{th:boundedness},
the functions $\alpha_t^*$ and $\beta_t^*$ are well-defined by
(\ref{eq:indicator*-1})--(\ref{eq:indicator*-2})   and
\begin{equation}\label{eq:indicator-functions}
\alpha_t^*(x)=\alpha_t(x),
\quad
\beta_t^*(x)=\beta_t(x)
\quad (x\in\R)
\end{equation}
for all $t\in\Gamma$. We say that the functions $\alpha_t^*$ and $\beta_t^*$
are the indicator functions of the triple $(\Gamma,p,w)$ at the point $t\in\Gamma$.
\subsection{Necessary conditions for Fredholmness}
The following necessary conditions for Fredholmness were obtained by the
author \cite[Theorem~8.1]{Karlovich03}.
\begin{theorem}\label{th:Fredholmness-necessity}
Let $\Gamma$ be a Carleson Jordan curve and let $p:\Gamma\to(1,\infty)$
be a continuous function satisfying the Dini-Lipschitz condition
{\rm(\ref{eq:Dini-Lipschitz})}. Suppose $w:\Gamma\to[0,\infty]$ is an
arbitrary weight such that the operator $S$ is bounded on $L^{p(\cdot)}(\Gamma,w)$.
If the operator $aP+Q$, where $a\in PC(\Gamma)$, is Fredholm on the
weighted Nakano space $L^{p(\cdot)}(\Gamma,w)$, then $a(t\pm 0)\ne 0$
and
\[
-\frac{1}{2\pi}\arg\frac{a(t-0)}{a(t+0)}+\frac{1}{p(t)}+
\theta\alpha_t\left(\frac{1}{2\pi}\log\left|\frac{a(t-0)}{a(t+0)}\right|\right)
+
(1-\theta)\beta_t\left(\frac{1}{2\pi}\log\left|\frac{a(t-0)}{a(t+0)}\right|\right)
\notin\Z
\]
for all $\theta\in[0,1]$ and all $t\in\Gamma$.
\end{theorem}
\subsection{Wiener-Hopf factorization of local representatives}
Fix $t\in\Gamma$. For 
$a\in PC(\Gamma)$ such that $a^{-1}\in L^\infty(\Gamma)$,
we construct a ``canonical'' function $g_{t,\gamma}$ which is locally equivalent
to $a$ at the point $t\in\Gamma$. The interior and the exterior of the unit circle
can be conformally mapped onto $D^+$ and $D^-$ of $\Gamma$, respectively,
so that the point $1$ is mapped to $t$, and the points $0\in D^+$ and
$\infty\in D^-$ remain fixed. Let $\Lambda_0$ and $\Lambda_\infty$
denote the images of $[0,1]$ and $[1,\infty)\cup\{\infty\}$ under this map.
The curve $\Lambda_0\cup\Lambda_\infty$ joins $0$ to $\infty$ and
meets $\Gamma$ at exactly one point, namely $t$. Let $\arg z$ be a
continuous branch of argument in $\mathbb{C}\setminus(\Lambda_0\cup\Lambda_\infty)$.
For $\gamma\in\mathbb{C}$, define the function $z^\gamma:=|z|^\gamma e^{i\gamma\arg z}$,
where $z\in\mathbb{C}\setminus(\Lambda_0\cup\Lambda_\infty)$. Clearly, $z^\gamma$
is an analytic function in $\mathbb{C}\setminus(\Lambda_0\cup\Lambda_\infty)$. The
restriction of $z^\gamma$ to $\Gamma\setminus\{t\}$ will be denoted by
$g_{t,\gamma}$. Obviously, $g_{t,\gamma}$ is continuous and nonzero on
$\Gamma\setminus\{t\}$. Since $a(t\pm 0)\ne 0$, we can define
$\gamma_t=\gamma\in\mathbb{C}$ by the formulas
\begin{equation}\label{eq:local-representative}
{\rm Re}\,\gamma_t:=\frac{1}{2\pi}\arg\frac{a(t-0)}{a(t+0)},
\quad
{\rm Im}\,\gamma_t:=-\frac{1}{2\pi}\log\left|\frac{a(t-0)}{a(t+0)}\right|,
\end{equation}
where we can take any value of $\arg(a(t-0)/a(t+0))$, which implies that
any two choices of ${\rm Re}\,\gamma_t$ differ by an integer only.
Clearly, there is a constant $c_t\in\mathbb{C}\setminus\{0\}$ such that
$a(t\pm 0)=c_tg_{t,\gamma_t}(t\pm 0)$, which means that $a$ is locally
equivalent to $c_tg_{t,\gamma_t}$ at the point $t\in\Gamma$.

For $t\in\Gamma$ and $\gamma\in\C$, consider the weight
\[
\varphi_{t,\gamma}(\tau):=|(\tau-t)^\gamma|,
\quad
\tau\in\Gamma\setminus\{t\}.
\]

From \cite[Lemma~7.1]{Karlovich03} we get the following.
\begin{lemma}\label{le:fact-sufficiency}
Let $\Gamma$ be a Carleson Jordan curve and let $p:\Gamma\to(1,\infty)$ be a
continuous function. Suppose $w:\Gamma\to[0,\infty]$ is an arbitrary weight
such that the operator $S$ is bounded on $L^{p(\cdot)}(\Gamma,w)$. If, for
some $k\in\Z$ and $\gamma\in\C$, the operator
$\varphi_{t,k-\gamma}S\varphi_{t,\gamma-k}I$ is bounded on $L^{p(\cdot)}(\Gamma,w)$,
then the function $g_{t,\gamma}$ admits a Wiener-Hopf factorization on
$L^{p(\cdot)}(\Gamma,w)$.
\end{lemma}
\subsection{Sufficient conditions for Fredholmness}
The following result is one of the main ingredients of the proof of
Theorem~\ref{th:Fredholmness}. The idea of its proof is borrowed from
the proof of \cite[Proposition~7.3]{BK97}.
\begin{theorem}\label{th:Fredholmness-sufficiency}
Let $\Gamma$ be a Carleson Jordan curve
and $p:\Gamma\to(1,\infty)$ be a continuous function satisfying
the Dini-Lipschitz condition
{\rm(\ref{eq:Dini-Lipschitz})}. Suppose $t_1,\dots,t_n\in\Gamma$ are
pairwise distinct points and
$\psi_j:\Gamma\setminus\{t_j\}\to(0,\infty)$ are continuous
functions such that the functions $W_{t_j}\psi_j$ are regular and
conditions {\rm(\ref{eq:boundedness-condition})} are fulfilled for
all $j\in\{1,\dots,n\}$. If $a\in PC(\Gamma)$ is such that $a(t\pm 0)\ne 0$
and
\[
-\frac{1}{2\pi}\arg\frac{a(t-0)}{a(t+0)}+\frac{1}{p(t)}+
\theta\alpha_t^*\left(\frac{1}{2\pi}\log\left|\frac{a(t-0)}{a(t+0)}\right|\right)
+
(1-\theta)\beta_t^*\left(\frac{1}{2\pi}\log\left|\frac{a(t-0)}{a(t+0)}\right|\right)
\notin\Z
\]
for all $\theta\in[0,1]$ and all $t\in\Gamma$, then the operator $aP+Q$
is Fredholm on the Nakano space $L^{p(\cdot)}(\Gamma,w)$
with weight $w$ given by {\rm(\ref{eq:weight})}.
\end{theorem}
\begin{proof}
We will follow the proof of \cite[Theorem~4.5]{Karlovich09-IWOTA07}
and \cite[Theorem~2.2]{Karlovich10-IWOTA08}.
If $aP+Q$ is Fredholm on $L^{p(\cdot)}(\Gamma,w)$, then,
by Theorem~\ref{th:necessity}, $a(t\pm 0)$ for all $t\in\Gamma$. Fix
an arbitrary $t\in\Gamma$ and choose $\gamma=\gamma_t$ as in
(\ref{eq:local-representative}). Then the function $a$ is locally
equivalent to $c_tg_{t,\gamma_t}$ at the point $t\in\Gamma$, where
$c_t\in\C\setminus\{0\}$ is some constant. In this case the main condition
of the theorem has the form
\[
1/p(t)-{\rm Re}\,\gamma_t+
\theta\alpha_t^*(-{\rm Im}\,\gamma_t)+
(1-\theta)\beta_t^*(-{\rm Im}\,\gamma_t)\notin\Z
\quad\mbox{for all}\quad \theta\in[0,1].
\]
Therefore, there exists a number $k_t\in\Z$ such that
\[
0<1/p(t)+k_t-{\rm Re}\,\gamma_t+
\theta\alpha_t^*(-{\rm Im}\,\gamma_t)+
(1-\theta)\beta_t^*(-{\rm Im}\,\gamma_t)<1
\quad\mbox{for all}\quad \theta\in[0,1].
\]
In particular, if $\theta=1$, then
\begin{equation}\label{eq:Fredholmness-sufficiency-1}
0<1/p(t)+{\rm Re}(k_t-\gamma_t)+\alpha_t^*({\rm Im}(k_t-\gamma_t));
\end{equation}
if $\theta=0$, then
\begin{equation}\label{eq:Fredholmness-sufficiency-2}
1/p(t)+{\rm Re}(k_t-\gamma_t)+\beta_t^*({\rm Im}(k_t-\gamma_t))<1.
\end{equation}
Consider the weights $\omega_t(\tau):=|\tau-t|$ and
\[
w_t:=\varphi_{t,k_t-\gamma_t}w
=
\omega_t^{{\rm Re}(k_t-\gamma_t)}
\eta_t{{\rm Im}(k_t-\gamma_t)}
\psi_1\dots\psi_n.
\]
If $t\in\Gamma\setminus\{t_1,\dots,t_n\}$, then the weight $w=\psi_1\dots\psi_n$
has no singularity at $t$ and
\[
\varphi_{t,k_t-\gamma_t}
=
\omega_t^{{\rm Re}(k_t-\gamma_t)}
\eta_t^{{\rm Im}(k_t-\gamma_t)}
\]
is a continuous function on $\Gamma\setminus\{t\}$. If $t=t_j\in\{t_1,\dots,t_n\}$,
then the weight $w/\psi_j$ has no singularity at $t_j$ and
\[
\varphi_{t_j,k_{t_j}-\gamma_{t_j}}\psi_j
=
\omega_{t_j}^{{\rm Re}(k_{t_j}-\gamma_{t_j})}
\eta_{t_j}^{{\rm Im}(k_{t_j}-\gamma_{t_j})}\psi_j
\]
is a continuous function on $\Gamma\setminus\{t_j\}$. Thus, in both cases,
the weight $w_t$ is of the same form as the weight $w$.

It is easy to see that the function $W_t(\omega_t^{{\rm Re}(k_t-\gamma_t)})$ is regular
and submultiplicative and
\[
\alpha\big(W_t^0(\omega_t^{{\rm Re}(k_t-\gamma_t)})\big)
=
\beta\big(W_t^0(\omega_t^{{\rm Re}(k_t-\gamma_t)})\big)
=
{\rm Re}(k_t-\gamma_t)
\quad\mbox{for every}\quad t\in\Gamma.
\]
Then, by Lemma~\ref{le:WW-estimates}, the functions
$W_t(\varphi_{t,k_t-\gamma_t})$ and
$W_t^0(\varphi_{t,k_t-\gamma_t})$  are regular and
submultiplicative and
\begin{eqnarray}
\label{eq:Fredholmness-sufficiency-3}
\alpha\big(W_t^0(\varphi_{t,k_t-\gamma_t})\big)
&=&
{\rm Re}(k_t-\gamma_t)+\alpha\big(W_t(\eta_t^{{\rm Im}(k_t-\gamma_t)})\big)
\\
\nonumber
&=&
{\rm Re}(k_t-\gamma_t)+\alpha_t^*({\rm Im}(k_t-\gamma_t)),
\\
\label{eq:Fredholmness-sufficiency-4}
\beta\big(W_t^0(\varphi_{t,k_t-\gamma_t})\big)
&=&
{\rm Re}(k_t-\gamma_t)+\beta\big(W_t(\eta_t^{{\rm Im}(k_t-\gamma_t)})\big)
\\
\nonumber
&=&
{\rm Re}(k_t-\gamma_t)+\beta_t^*({\rm Im}(k_t-\gamma_t))
\end{eqnarray}
for all $t\in\Gamma\setminus\{t_1,\dots,t_n\}$. Analogously, if $t=t_j\in\{t_1,\dots,t_n\}$,
then the function $W_{t_j}(\varphi_{t_j,k_{t_j}-\gamma_{t_j}}\psi_j)$ is regular
and submultiplicative and
\begin{eqnarray}
\alpha\big(W_{t_j}^0(\varphi_{t_j,k_{t_j}-\gamma_{t_j}}\psi_j)\big)
&=&
{\rm Re}(k_{t_j}-\gamma_{t_j})+\alpha_{t_j}^*({\rm Im}(k_{t_j}-\gamma_{t_j})),
\label{eq:Fredholmness-sufficiency-5}
\\
\beta\big(W_{t_j}^0(\varphi_{t_j,k_{t_j}-\gamma_{t_j}}\psi_j)\big)
&=&
{\rm Re}(k_{t_j}-\gamma_{t_j})+\beta_{t_j}^*({\rm Im}(k_{t_j}-\gamma_{t_j})).
\label{eq:Fredholmness-sufficiency-6}
\end{eqnarray}
Combining relations (\ref{eq:Fredholmness-sufficiency-1})--(\ref{eq:Fredholmness-sufficiency-6})
with conditions (\ref{eq:boundedness-condition}), we see that,
by Theorem~\ref{th:boundedness}, the operator $S$ is bounded on
$L^{p(\cdot)}(\Gamma,w_t)=L^{p(\cdot)}(\varphi_{t,k_t-\gamma_t}w)$,
where $t\in\Gamma$. Therefore the operator $\varphi_{t,k_t-\gamma_t}S\varphi_{t,\gamma_t-k_t}I$
is bounded on $L^{p(\cdot)}(\Gamma,w)$.

Then, in view of Lemma~\ref{le:fact-sufficiency}, the function $g_{t,\gamma_t}$
admits a Wiener-Hopf factorization on $L^{p(\cdot)}(\Gamma,w)$. From
Theorem~\ref{th:factorization} we deduce that the operator $g_{t,\gamma_t}P+Q$
is Fredholm. It is not difficult to see that in this case the operator
$c_tg_{t,\gamma_t}P+Q$ is also Fredholm. Thus, for all local representatives
$c_tg_{t,\gamma_t}$ of the coefficient $a$, the operators $c_tg_{t,\gamma_t}P+Q$
are Fredholm. To finish the proof, it remains to apply the local principle
(Theorem~\ref{th:local_principle}), which says that the operator $aP+Q$
is Fredholm.
\end{proof}
\subsection{Proof of Theorem~{\rm\ref{th:Fredholmness}}}
\begin{proof}
\textit{Necessity.}
If $aP+bQ$ is Fredholm, then $a^{-1},b^{-1}\in L^\infty(\Gamma)$
by Theorem~\ref{th:necessity}. Put $c:=a/b$. Then $c(t\pm 0)\ne 0$
for all $t\in\Gamma$. Further, the operator $bI$ is invertible
on $L^{p(\cdot)}(\Gamma,w)$. Therefore, the operator
\[
cP+Q=(bI)^{-1}(aP+bQ)
\]
is Fredholm. From Theorem~\ref{th:Fredholmness-necessity}
and equalities (\ref{eq:indicator-functions}) it follows that
\begin{eqnarray}
\label{eq:Fredholmness-1}
&&
-\frac{1}{2\pi}\arg\frac{c(t-0)}{c(t+0)}+\frac{1}{p(t)}
\\
\nonumber
&&+
\theta\alpha_t^*\left(\frac{1}{2\pi}\log\left|\frac{c(t-0)}{c(t+0)}\right|\right)
+
(1-\theta)\beta_t^*\left(\frac{1}{2\pi}\log\left|\frac{c(t-0)}{c(t+0)}\right|\right)
\notin\Z
\end{eqnarray}
for all $\theta\in [0,1]$ and all $t\in\Gamma$. The latter condition in conjunction
with $c(t\pm 0)\ne 0$ for all $t\in\Gamma$ is equivalent to
\[
0\notin\bigcup_{t\in\Gamma}\cL(c(t-0),c(t+0);p(t),\alpha_t^*,\beta_t^*).
\]
Thus, the function $c=a/b$ is $L^{p(\cdot)}(\Gamma,w)$-nonsingular.
Necessity is proved.

\textit{Sufficiency.}
The $L^{p(\cdot)}(\Gamma,w)$-nonsingularity of $c=a/b$ implies that
$c(t\pm 0)\ne 0$ and (\ref{eq:Fredholmness-1}) holds for all $\theta\in[0,1]$
and all $t\in\Gamma$. Then the operator $cP+Q$ is Fredholm by
Theorem~\ref{th:Fredholmness-sufficiency}. Since $\inf\limits_{t\in\Gamma}|b(t)|>0$,
we see that the operator $bI$ is invertible. Thus, the operator
\[
aP+bQ=(bI)(cP+Q)
\]
is Fedholm.
\end{proof}

\end{document}